\documentclass[12pt]{article}   
\usepackage{graphicx,psfrag}
\usepackage{color}
\usepackage{array}
\usepackage{colortbl}
\usepackage{tabularx}
\newtheorem{theorem}{Theorem} 
\newtheorem{corollary}[theorem]{Corollary}

\newtheorem{lemma}[theorem]{Lemma}

\newtheorem{example}[theorem]{Example}

\newenvironment {proof} {{\it
Proof.}}{\hspace*{\fill}$\Box$\par\vspace{4mm}}
\usepackage{amsmath}
\usepackage{amsfonts}
\newfont{\bb}{msbm10}

\def\:{\! :\!}

\newcommand\0{{\bf 0}}

\usepackage{tikz}

\begin{document}

\title{Symmetric, Hankel-symmetric, and Centrosymmetric Doubly Stochastic Matrices}

 \author{Richard A. Brualdi\\
 Department of Mathematics\\
 University of Wisconsin\\
 Madison, WI 53706\\
 {\tt brualdi@math.wisc.edu}
 \and
 Lei Cao\\
 Department of Mathematics\\
 Georgian Court University\\
Lakewood, New Jersey 08701\\
 {\tt leicaomath@gmail.com}
 }

\maketitle

 \begin{abstract} 
We investigate convex polytopes of doubly stochastic matrices having special structures: symmetric, Hankel symmetric, centrosymmetric, and both symmetric and Hankel symmetric. We determine dimensions of these polytopes and classify their extreme points. We also determine a basis of the real vector spaces generated by permutation matrices with these special structures.


\end{abstract}

\section{Introduction}
Let $n$ be a positive integer. Let ${\mathcal S}_n$ be the set of permutations of $\{1,2,\ldots,n\}$, and  let ${\mathcal P}_n$ be the set of $n\times n$ permutation matrices.
The permutation $\sigma=(j_1,j_2,\ldots,j_n)$ is associated in the usual way to the permutation matrix $P=[p_{ij}]$ where $p_{1j_1}=p_{2j_2}=\ldots=p_{nj_n}=1$, and all other $p_{ij}=0$. We usually identify a permutation in ${\mathcal S}_n$ and its corresponding permutation matrix in ${\mathcal P}_n$.

The convex hull of ${\mathcal P}_n$ is the well-studied polytope $\Omega_n$ of dimension $(n-1)^2$  consisting of all the $n\times n$ {\it doubly stochastic matrices}, that is, matrices with nonnegative entries having all row and column sums equal to 1. Since a permutation matrix $P$ cannot be expressed as a convex combination of permutation matrices different from $P$, the set ${\mathcal E}(\Omega_n)$ of extreme points of $\Omega_n$ is ${\mathcal P}_n$. Now let $\Omega_n^t$ denote the set of $n\times n$ symmetric doubly stochastic matrices\footnote{We use the superscript `t' to reflect the fact that symmetric matrices are invariant under transposition.}. Then $\Omega_n^t$ is also a convex polytope and has dimension equal to $n\choose 2$. Its set ${\mathcal E}(\Omega_n^t)$ of extreme points includes the set ${\mathcal P}_n^t$ of  $n\times n$ symmetric permutation matrices but there are other extreme points as well. In fact, ${\mathcal E}(\Omega_n^t)$ consists of all $n\times n$ matrices which, after simultaneous permutations of its rows and columns, that is, replacing a matrix $X$ with $PXP^t$ for some permutation matrix $P$, are direct sums of matrices of three types:
\begin{equation}
\label{eq:T_2}
I_1=\left[\begin{array}{c} 1\end{array}\right], T_2=\left[\begin{array}{cc} 0&1\\ 1&0\end{array}\right],\end{equation}
and, with $k\ge 3$ odd,  the $k\times k$ matrices
\begin{equation}\label{eq:C_k}
C_k=\frac{1}{2}\left[\begin{array}{c|c|c|c|c|c|c}
&1&&&&&1\\  \hline
1&&1&&&&\\  \hline
&1&&1&&&\\ \hline
&&1&&\ddots&&\\ \hline
&&&\ddots&&\ddots&\\ \hline
&&&&1&&1\\ \hline
1&&&&&1&\end{array}\right].\end{equation}
Here, as later, entries in blank positions are zero. These and other facts about doubly stochastic matrices can be found in Chapter 9 of \cite{RAB}. With an $n\times n$ symmetric matrix $S=[s_{ij}]$ there is associated a {\it loopy graph} $G(S)$ with vertex set $V=\{1,2,\ldots,n\}$ and an edge joining vertices $i$ and $j$ exactly when $s_{ij}\ne 0$. Note that if $s_{ii}\ne 0$, then there is a {\it loop} at vertex $i$, that is, an edge joining vertex $i$ to itself.  The loopy graphs of the extreme points of $\Omega_n^t$ are characterized by having connected components which are either loops (corresponding to $I_1$ in (\ref{eq:T_2})), edges (corresponding to $T_2$ in (\ref{eq:T_2})), or  odd cycles of length $k\ge 3$ (corresponding to $C_k$ in (\ref{eq:C_k})). Such loopy graphs determine a unique extreme point of $\Omega_n^t$. The following lemma follows from our discussion above but it can also be used to characterize the extreme points of $\Omega_n^t$.  We outline the direct argument
\cite{AC}  since we shall generalize it later,

\begin{lemma}\label{lem:symext}
 All the extreme points of $\Omega_n^t$ are of the form $\frac{1}{2}(P+P^t)$ 
where $P$ is a permutation matrix. If $\frac{1}{2}(P+P^t)$  is not an extreme point of $\Omega_n^t$, then there exist  symmetric permutation matrices $Q_1$
and $Q_2$  such that $P+P^t=Q_1+Q_2$.\hfill{$\Box$}
\end{lemma}

\begin{proof}
Let $A\in \Omega_n^t$. Since $A\in \Omega_n$, $A$ is a convex combination of $n\times n$ permutation matrices, $A=\sum_{i=1}^kc_iP_i$ with the $c_i> 0$ and $\sum_{i=1}^k c_i=1$. Taking the transpose of this equation and using the fact that $A$ is symmetric, we get
\[A=\sum_{i=1}^kc_i \left(\frac{1}{2}(P_i+P_i^t)\right).\]
Thus all the extreme points of $\Omega_n^t$ are of the form $\frac{1}{2}(P+P^t)$ for some permutation matrix $P$. If $P$ has an even permutation cycle of length at least 4, then it is easy to see that $\frac{1}{2}(P+P^t)$  is not an extreme point since then there exist symmetric permutation matrices $Q_1$ and $Q_2$ such that $P+P^t =Q_1+Q_2$ (see the next example). If all cycles of $P$ have length 2 or are odd, then it is easily checked that $\frac{1}{2}(P+P^t)$  is an  extreme point (see (\ref{eq:C_k})).
\end{proof}

\begin{example}\label{ex:sym-even}{\rm
If $n=6$ and 
\[P=\left[\begin{array}{c|c|c|c|c|c}
&1&&&&\\ \hline
&&1&&&\\ \hline
&&&1&&\\ \hline
&&&&1&\\ \hline
&&&&&1\\ \hline
1&&&&&\end{array}\right],\mbox{ then }\]
\[P+P^t=\left[\begin{array}{c|c|c|c|c|c}
&1&&&&1\\ \hline
1&&1&&&\\ \hline
&1&&1&&\\ \hline
&&1&&1&\\ \hline
&&&1&&1\\ \hline
1&&&&1&\end{array}\right]=
\left[\begin{array}{c|c|c|c|c|c}
&1&&&&\\ \hline
1&&&&&\\ \hline
&&&1&&\\ \hline
&&1&&&\\ \hline
&&&&&1\\ \hline
&&&&1&\end{array}\right]+
\left[\begin{array}{c|c|c|c|c|c}
&&&&&1\\ \hline
&&1&&&\\ \hline
&1&&&&\\ \hline
&&&&1&\\ \hline
&&&1&&\\ \hline
1&&&&&\end{array}\right].\]\hfill{$\Box$}

}\end{example}

The polytope $\Omega_n^t$  has dimension $n\choose 2$ (entries above the main diagonal of an $n\times n$ symmetric matrix can be chosen arbitrarily but small and this results in  a unique matrix in  $\Omega_n^t$), and this implies that the real {\it vector space ${\Re}({\mathcal P}_n^t)$ spanned by the symmetric permutation matrices} has dimension equal to ${n\choose 2}+1$. A basis of this vector space consists of the $n\choose 2$ $n\times n$ permutation matrices corresponding to transpositions along with the identity matrix $I_n$.

Symmetric matrices $A$ are those matrices that are invariant under transposition ($A^t=A$), that is, matrices invariant under a reflection about the main diagonal. In \cite{BF} an analogous property, invariance under a reflection about the  so-called antidiagonal, renamed the {\it Hankel-diagonal}\ \footnote{The motivation for calling this the Hankel-diagonal
was that Hankel matrices are constant on the antidiagonal and on all diagonals parallel to it.  In contrast, Toeplitz matrices are matrices constant on the main diagonal and all diagonals parallel to it. We use $A^h$ to denote the matrix obtained from a square matrix $A$ by reflection about the Hankel-diagonal and are tempted to think of the main diagonal as the Toeplitz diagonal and the  `t' in $A^t$ to stand for Toeplitz.}, was considered in \cite{BF}. Thus the
 {\it Hankel diagonal} of an $n\times n$ matrix $A=[a_{ij}]$ is the set of positions $\{(i,n+1-i):1\le i\le n\}$. The {\it Hankel transpose} $A^h=[a_{ij}^h]$ of an $n\times n$ matrix $A$ satisfies  $a_{ij}^h=a_{n+1-j, n+1-i}$ for all $i$ and $j$. The matrix $A$ is called {\it Hankel symmetric}\footnote{Also called {\it persymmetric} but we prefer Hankel symmetric.} provided $A^h=A$, that is, 
$a_{ij}=a_{n+1-j, n+1-i}$ for all $i$ and $j$. The set of $n\times n$ Hankel symmetric permutation matrices is denoted by ${\mathcal P}_n^h$.  We denote the $n\times n$ permutation matrix whose 1's are on the Hankel-diagonal by $L_n$ and call it the {\it Hankel identity matrix}.

\begin{example}\label{ex:hankelsym}{\rm
The following matrices are Hankel symmetric permutation matrices:
\[P=\left[\begin{array}{ccc}
0&1&0\\
0&0&1\\
1&0&0\end{array}\right] \mbox{ and }
Q=\left[\begin{array}{cccc}
0&1&0&0\\
1&0&0&0\\
0&0&0&1\\
0&0&1&0\end{array}\right].\]
The matrix $Q$ is also symmetric: $Q^t=Q^h=Q$. As seen with $Q$,  an entry off both the main diagonal and Hankel diagonal of a symmetric, Hankel symmetric matrix determines three other entries; an entry on the main diagonal or Hankel diagonal determines one other entry, unless it is on both diagonals (in which case the matrix has odd order $n$ and the entry is in the {\it central position} $((n+1)/2,(n+1)/2)$).
\hfill{$\Box$}
}\end{example}

The set of Hankel symmetric doubly stochastic matrices also forms a
convex polytope  $\Omega_n^h$  of dimension $n\choose 2$,  the characterization of  whose set ${\mathcal E}(\Omega_n^h)$ of extreme points follows easily from that of 
${\mathcal E}(\Omega_n^t)$. The matrix $C_k$ in (\ref{eq:C_k}) is also Hankel symmetric. After a {\it simultaneous Hankel permutation} of the rows and columns (that is, replacing a matrix $X$ with $PXP^h$ for some permutation matrix $P$ and its Hankel transpose $P^h$), the extreme points of $\Omega_n^h$ are the {\it Hankel direct sums} (direct sums with respect to the Hankel diagonal) of matrices $I_1$, $I_2$ (the Hankel counterpart to $T_2$), and $C_k$ for odd $k\ge 3$. The analogue of  Lemma \ref{lem:symext} for Hankel symmetric doubly stochastic matrices is given next, and it follows from that lemma using the above observations.

\begin{lemma}\label{lem:hankelext}
 All the extreme points of $\Omega_n^h$ are of the form $\frac{1}{2}(P+P^h)$ 
where $P$ is a permutation matrix. If $\frac{1}{2}(P+P^h)$  is not an extreme point of $\Omega_n^h$, then there exist  Hankel symmetric permutation matrices $Q_1$
and $Q_2$  with $P+P^h=Q_1+Q_2$.\hfill{$\Box$}
\end{lemma}

As with $\Omega_n^t$, the polytope $\Omega_n^h$  has dimension $n\choose 2$ implying that the vector space ${\mathbb Q}({\mathcal P}_n^t)$ spanned by the Hankel symmetric permutation matrices has dimension equal to ${n\choose 2}+1$. A basis of this vector space consists of the $n\times n$ permutation matrices corresponding to {\it Hankel transpositions} (transpositions about the Hankel diagonal) along with the Hankel identity matrix $L_n$.

There is a third, and related, invariance property that we consider. Let $A=[a_{ij}]$ be an $n\times n$ matrix. Then $A^t$ is a reflection of  $A$ about the main diagonal ($(i,j)\rightarrow (j,i)$), and $A^h$ is a reflection of $A$ about the Hankel-diagonal ($(i,j)\rightarrow (n+1-j,n+1-i)$). Thus $(A^t)^h$
is a rotation of $A$ by 180 degrees ($(i,j)\rightarrow (n+1-i,n+1-j)$). We denote the matrix obtained from $A$ by a rotation of 180 degrees by $A^{\pi}$. A matrix $A$ for which $A^{\pi}=A$,  that is, $A$ is  invariant under a 180 degree rotation, is called {\it centrosymmetric}.
If a matrix is both symmetric and Hankel symmetric, then it is also centrosymmetric.  If $P_{\sigma}=[p_{ij}]$ is the permutation matrix corresponding to a permutation $\sigma$ of $\{1,2,\ldots,n\}$, then $P_{\sigma}$ and $\sigma$ are each called  {\it centrosymmetric} provided
$\sigma(i)+\sigma(n+1-i)=n+1$ for $1\le i\le n$, equivalently, $p_{ij}=p_{n+1-i,n+1-j}$ for $1\le i,j\le n$. Let ${\mathcal S}_n^{\pi}$ denote the set of centrosymmetric permutations in ${\mathcal S}_n$, and let ${\mathcal P}_n^{\pi}$ denote the corresponding set of $n\times n$  centrosymmetric permutation matrices in ${\mathcal P}_n$.  

\begin{example}{\rm \label{ex:extra}
Let $n=4$. An example of a centrosymmetric permutation matrix
which is neither symmetric nor Hankel symmetric is
\[\left[\begin{array}{c|c|c|c}
&1&&\\ \hline
&&&1\\ \hline
1&&&\\ \hline
&&1&\end{array}\right].\]
}
\hfill{$\Box$}\end{example}

This paper is organized as follows. In the next section we investigate properties of   the set ${\mathcal P}_n^{\pi}$ of $n\times n$ centrosymmetric permutation matrices and the convex polytope ${\Omega}_n^{\pi}$ of $n\times n$ {\it centrosymmetric doubly stochastic matrices}. In the subsequent section we are concerned with the set ${\mathcal P}_n^{t\& h}$ 
of $n\times n$ {\it symmetric and Hankel symmetric permutation matrices} and the convex polytope ${\Omega}_n^{t\& h}$ of $n\times n$ {\it symmetric and Hankel-symmetric doubly stochastic matrices}. We have that  ${\Omega}_n^{t\& h}$  forms a  subpolytope  of $\Omega_n^{\pi}$ which also forms a subpolytope of $\Omega_n$.

Our paper is partly expository in nature in that, for the convenience of the reader, we attempt to collect the relevant facts that motivated our investigations.

\section{Centrosymmetric Matrices}

Let $P=[p_{ij}]$ be an $n\times n$ centrosymmetric permutation matrix. If $n$ is even, then $P$ is determined by its $\frac{n}{2}\times n$ submatrix formed by its first $\frac{n}{2}$ rows; moreover, in this submatrix, there is exactly one 1 in the union of columns $i$ and $n+1-i$ for each $i=1,2,\ldots,\frac{n}{2}$. It  follows that when $n=2m$  is even, the number of $n\times n$ centrosymmetric permutation matrices equals $2^mm!$ \cite{BBS}. If $n=2m+1$ is odd, then an $n\times n$ centrosymmetric permutation matrix has a 1 in the central position $(m+1,m+1)$, and deleting row and column $m+1$ results in a $2m\times 2m$ centrosymmetric permutation matrix. Thus the number of centrosymmetric permutation matrices when $n=2m+1$ is also $2^mm!$.

We collect the  following elementary properties which are easily verified:
\begin{itemize}
\item[\rm(i)] $(A^{t})^t=(A^{h})^h=(A^{\pi})^{\pi}=A$
\item[\rm (ii)] $A^{\pi}=(A^{t})^h=(A^{h})^t$.
\item[\rm (iii)] $(A^{\pi})^t=A^h$, $(A^{\pi})^h=A^t$.
\item[\rm (iv)] Any two of $A^t=A$, $A^h=A$, and $A^{\pi}=A$ implies the other.
\item[\rm (v)] $A^t=A^h$ if and only if $A^{\pi}=A$, that is, a matrix is centrosymmetric if and only if its transpose equals its Hankel-transpose.
\item[\rm (vi)] If $A$ and $B$  are $n\times n$ matrices, then $(AB)^{\pi}=B^{\pi}A^{\pi}$. In particular, if $P$ is a centrosymmetric permutation matrix, then $(PAP^{\pi})^{\pi} =PA^{\pi}P^{\pi}$, and thus if $A$ is centrosymmetric, then $PAP^{\pi}$ is also centrosymmetric. We say that such a matrix $PAP^{\pi}$ is obtained from $A$ by a {\it simultaneous centrosymmetric-permutation} of rows and columns. 
\item[\rm (vii)]  Let $n=2m$ be even. Then 
${\mathcal P}_{2m}^{\pi}$ is a subgroup of the group   ${\mathcal P}_{2m}$ of $2m\times 2m$ permutation matrices under multiplication and is  isomorphic to the {\it hyperoctahedral group} ${\mathcal B}_{m}$. Recall that ${\mathcal B}_{m}$ is the group of {\it signed permutations} 
defined to be the permutations $\rho$ of $\{\pm 1, \pm 2, \ldots, \pm m\}$ such that $\rho(-i)=-\rho(i)$ for $i=1,2,\ldots,m$. This group can be identified with the multiplicative group ${\mathcal P}_m^{\pm}$ of $m\times m$ {\it signed permutation matrices}, that is, matrices $DP$ where $D$ is an $m\times m$ diagonal matrix with $\pm 1$'s on the main diagonal and $P$ is a  permutation matrix.  An isomorphism is given by
$\theta: {\mathcal S}_{2m}^{\pi}\rightarrow {\mathcal B}_{m}$ where for $\sigma\in {\mathcal S}_{2m}^{\pi}$ and $1\le i\le m$,
 \[\theta(\sigma)(i)=\left\{
\begin{array}{ll} \sigma(m+i)-m & \mbox{ if $\sigma(m+i)>m$}\\
\sigma(m+i)-m-1,&  \mbox{ otherwise.}
\end{array}\right. \]
In terms of our centrosymmetric permutation matrices, we have
$\theta:  {\mathcal P}_{2m}^{\pi}\rightarrow {\mathcal B}_m$ where
  $\theta(P)$ is the $m\times m$ signed permutation matrix $[q_{ij}]$  defined by 
$q_{ij}=1$ if $\theta(\sigma)(i)$ is positive and $q_{ij}=-1$ if $\theta(\sigma)(i)$ is negative.
 For example, if $\sigma\in {\mathcal S}_n^{\pi}$ is $(2,1,6,5,4,3,8,7)$, we have
\[P=\left[\begin{array}{c|c|c|c|c|c|c|c}
&1&&&&&&\\ \hline
1&&&&&&&\\ \hline
&&&&&1&&\\ \hline
&&&&1&&&\\ \hline
&&&1&&&&\\ \hline
&&1&&&&&\\ \hline
&&&&&&&1\\ \hline
&&&&&&1&\end{array}\right]\rightarrow
\theta(P)=\left[\begin{array}{r|r|r|r}
-1&&&\\ \hline
&-1&&\\ \hline
&&&1\\ \hline
&&1&\end{array}\right].\]
See e.g. section 3 of  \cite{BBS}.
\end{itemize}

The {\it term rank} $\rho(A)$ of an $n\times n$  $(0,1)$-matrix $A$ equals the maximum cardinality of a set of $1$'s of $A$ with no two of the $1$'s from the same row or column. By the K\H onig-Eg\`ervary theorem, $\rho(A)$ equals the minimum number $\beta(A)$ of rows and columns of $A$ that contain all the 1's of $A$. We define a set of $1$'s of $A$ to be {\it $\pi$-invariant} provided the set is invariant under a rotation by 180 degrees.  Similarly, we define a set of rows and columns of $A$ to be {\it $\pi$-invariant} provided the set is invariant under a rotation by 180 degrees. The {\it centrosymmetric term rank} $\rho_{\pi}(A)$ is  defined to be the maximum cardinality of a $\pi$-invariant set of $1$'s of $A$ with no two of the $1$'s from the same row or column. A {\it centrosymmetric cover} of $A$ is a $\pi$-invariant set of rows and columns that contain all the 1's of $A$. The minimum cardinality of a centrosymmetric cover of $A$ is denoted by $\beta_{\pi}(A)$.
In investigating $\rho_{\pi}(A)$ there is no loss in generality in assuming that $A$ is centrosymmetric.

First assume that $n=2m$ is even and $A$ is a $2m\times 2m$ centrosymmetric $(0,1)$-matrix. Thus $A$ has the form given by
\begin{equation}\label{eq:centroform}
A=\left[\begin{array}{c|c}
A_1&A_2\\ \hline
A_2^{\pi}&A_1^{\pi}\end{array}\right].\end{equation}
Let $\rho_{*}([A_1\ | \  A_2])$ equal the maximum cardinality of a set of $1$'s of $[A_1\ | \  A_2]$ with no two from the same row and no two from the union of columns $i$ and $(2m+1-i)$.

\begin{theorem}\label{th:centrotermrank}
Let $A$ be a  $2m\times 2m$ centrosymmetric $(0,1)$-matrix. Then there exists
a centrosymmetric permutation matrix $P\le A$ $($entrywise$)$ if and only
if there does not exist a centrosymmetric cover of $A$ of size $<2m$. More generally,  $\rho_{\pi}(A)= \beta_{\pi}(A)$.
\end{theorem}

\begin{proof} 
The matrix $A$ has the form given in $(\ref{eq:centroform})$.
Let $A_2'$ be the matrix obtained from $A_2$ by reversing the order of its columns.
Let $B$ be the $m\times m$ $(0,1)$-matrix $A_1+_b A_2'$
where $+_b$ denotes {\it Boolean sum} (so $0+_b0=0, 0+_b1=1,1+_b1=1$). Then $\rho(B)=\rho_{*}([A_1\ | \  A_2])$ and
$\rho_{\pi}(A)=2\rho(B)$. Moreover, it follows from our discussion above that 
 $\rho(B)=\beta(B)$. Since $A$ is centrosymmetric, we obtain that $\beta_{\pi}(A)=2\beta(A)$. Hence $\rho_{\pi}(A)=\beta_{\pi}(A)$.
\end{proof}

\begin{example}\label{ex:double}{\rm To illustrate the proof of Theorem \ref{th:centrotermrank}, consider the $8\times 8$ centrosymmetric matrix 
\[A=\left[\begin{array}{c|c}
A_1&A_2\\ \hline
A_2^{\pi}&A_1^{\pi}\end{array}\right]=\left[\begin{array}{cccc|cccc}
1&1&0&1&0&0&0&1\\
0&0&0&0&1&1&0&0\\
0&0&0&0&0&1&0&0\\
1&0&1&0&0&1&0&0\\ \hline
0&0&1&0&0&1&0&1\\
0&0&1&0&0&0&0&0\\
0&0&1&1&0&0&0&0\\
1&0&0&0&1&0&1&1\end{array}\right].\]
Then, using shading to denote the special 1's, we have that
\[B=A_1+_bA_2'=\left[\begin{array}{cccc}
1&\cellcolor[gray]{0.8}1&0& 1\\
0&0&  1&\cellcolor[gray]{0.8} 1\\
		0&0&\cellcolor[gray]{0.8} 1&0\\
	\cellcolor[gray]{0.8}1&1&1&0\end{array}\right]\rightarrow
[A_1\ A_2]=\left[\begin{array}{cccc|cccc}
	1&\cellcolor[gray]{0.8} 1&0&1&0&0&0&1\\
0&0&0&0&\cellcolor[gray]{0.8} 1&1&0&0\\
0&0&0&0&0&\cellcolor[gray]{0.8}  1&0&0\\
\cellcolor[gray]{0.8}1&0&1&0&0&1&0&0\end{array}\right],\]
resulting in  the centrosymmetric permutation matrix in $A$ given by
\[\left[\begin{array}{cccc|cccc}
1&\cellcolor[gray]{0.8} 1&0&1&0&0&0&1\\
0&0&0&0&\cellcolor[gray]{0.8}1&1&0&0\\
0&0&0&0&0&\cellcolor[gray]{0.8} 1&0&0\\
\cellcolor[gray]{0.8} 1&0&1&0&0&1&0&0\\ \hline
0&0&1&0&0&1&0&\cellcolor[gray]{0.8}1\\
0&0&\cellcolor[gray]{0.8} 1&0&0&0&0&0\\
0&0&1&\cellcolor[gray]{0.8} 1&0&0&0&0\\
1&0&0&0&1&0&\cellcolor[gray]{0.8} 1&1\end{array}\right].\]

}
\hfill{$\Box$}
\end{example}

Now assume that $n=2m+1$ and let $A$ be a $(2m+1)\times (2m+1)$ centrosymmetric $(0,1)$-matrix. Then $A$ has the form
\begin{equation}\label{eq:centro-odd}
A=\left[\begin{array}{c|c|c}
A_1&u&A_2\\ \hline
v&x&v^{\pi}\\ \hline
A_2^{\pi}&u^{\pi}&A_1^{\pi}\end{array}\right]\end{equation}
where $A_1$ and $A_2$ are $m\times m$ matrices, $u$ is $m\times 1$, $v$ is $1\times m$ and $x$ is $1\times 1$. A  $\pi$-invariant set of $1$'s of $A$ with no two on the same row or column cannot contain any 1's from $u$ or $v$.
Thus $\rho_{\pi} (A)=\rho_{\pi}(A')+x$ where $A'$ is the $2m\times 2m$ matrix obtained from $A$ by deleting row $(m+1)$ and column $(m+1)$.
In this odd case, we may have that
$\beta_{\pi}(A)>\rho_{\pi}(A)$. For instance, with the centrosymmetric matrix
\[A=\left[\begin{array}{ccc}
0&1&0\\ 1&0&1\\ 0&1&0\end{array}\right],\] we  have that
$\rho_{\pi}(A)=0$ but $\beta_{\pi}(A)=\beta(A)=2$. But we do have the following corollary.

\begin{corollary}\label{cor:centrotermrank}
Let $A$ be a  $(2m+1)\times (2m+1)$ centrosymmetric $(0,1)$-matrix as in $(\ref{eq:centro-odd})$.  Then there exists
a centrosymmetric permutation matrix $P\le A$ $($entrywise$)$ if and only $x=1$ and the matrix $A'$ obtained from $A$ by deleting row and column $m+1$ satisfies $ \beta_{\pi}(A')=2m$.
\end{corollary}

We now turn   to the convex polytope   ${\Omega}_n^{\pi}$ of $n\times n$ centrosymmetric doubly stochastic matrices (a subpolytope of $\Omega_n$) and its set ${\mathcal E}(\Omega_n^{\pi})$ of extreme points which necessarily includes the $n\times n$ centrosymmetric permutation matrices. The polytope $\Omega_n^{\pi}$ was investigated in \cite{AC} and the results are summarized in the next theorem.
For the convenience of the reader we briefly outline the proof.

\begin{theorem}\label{th:centroext}
The extreme points of $\Omega_n^{\pi}$ are  of the form $\frac{1}{2}(P+P^{\pi})$ where $P$ is a permutation matrix. If $n$ is even and $P$ is not centrosymmetric, then the matrix  $\frac{1}{2}(P+P^{\pi})$ is not an extreme point of  $\Omega_n^{\pi}$; in fact,  there exist centrosymmetric 
permutation matrices $Q_1$ and $Q_2$ such that $P+P^{\pi}=Q_1+Q_2$.
Thus if $n$ is even, the set ${\mathcal E}_n(\Omega_n^{\pi})$  of extreme points of $\Omega_n^{\pi}$ is the set ${\mathcal P}_n^{\pi}$ of $n\times n$ centrosymmetric permutation matrices. If $n\ge 3$ is odd, then ${\mathcal P}_n^{\pi}$ is  proper subset of ${\mathcal E}(\Omega_n^{\pi})$.
\end{theorem}

\begin{proof} Let $A\in \Omega_n^{\pi}$. There  are permutation matrices $P_1,P_2,\ldots,P_m$ and positive numbers $c_1,c_2,\ldots,c_m$ with $\sum_{i=1}^mc_i=1$ such that
\[A=\sum_{i=1}^mc_iP_i\mbox{ and }A^{\pi}=\sum_{i=1}^mc_iP_i^{\pi}.\]
Since $A$ is centrosymmetric, 
\[A=\sum_{i=1}^m c_i\left(\frac{1}{2}(P_i+P_i^{\pi})\right)\]from which the first assertion follows.

Now assume that $n$ is even, and consider $P+P^{\pi}$ where $P$ is an $n\times n$ permutation matrix and $P\ne P^{\pi}$. Then $P+P^{\pi}$ is a centrosymmetric $(0,1,2)$-matrix with all row and column sums equal to 2. 
By  centrosymmetry, each $2\times 2$ submatrix of the form 
$(P+P^{\pi})[\{i,n+1-i\};\{j,n+1-j\}]$ equals one of
\begin{equation}\label{eq:reverse}
O_2=\left[\begin{array}{cc} 0&0 \\ 0&0\end{array}\right],
I_2=\left[\begin{array}{cc} 1&0 \\ 0&1\end{array}\right],
L_2=\left[\begin{array}{cc} 0&1 \\ 1&0\end{array}\right].\end{equation}
We construct an $\frac{n}{2}\times \frac{n}{2}$ $(0,1)$-matrix $X$ by replacing each such $2\times 2$ submatrix with a single 0 or 1 as follows:
\begin{equation}\label{eq:back}
O_2\rightarrow 0, I_2\rightarrow 1, L_2\rightarrow 1.\end{equation}
 Each row and column sum of $X$ equals 2, and hence $X$ is the sum of two permutation matrices $R_1$ and $R_2$. Reversing the arrows in (\ref{eq:back}) (in case of a 1 we have two choices) we can obtain two centrosymmetric permutation matrices $Q_1$ and $Q_2$ such that $P+P^{\pi}=Q_1+Q_2$.  Since every centrosymmetric matrix is an extreme point of $\Omega_n^{\pi}$,  we now conclude that ${\mathcal E}_n(\Omega_n^{\pi})={\mathcal P}_n^{\pi}$ if $n$ is even. For $n$ odd, see the next example.
\end{proof}

\begin{example}\label{ex:cent}{\rm Let $n=3$.
We claim that the centrosymmetric doubly stochastic matrix
\[A=\frac{1}{4}\left[\begin{array}{ccc}
3&1&0\\
1&2&1\\
0&1&3\end{array}\right]\]
is not in the convex hull of ${\mathcal P}_3^{\pi}$. There are only two $3\times 3$  centrosymmetric permutation matrices, namely the identity matrix $I_3$ and the Hankel-identity matrix $L_3$. The only other extreme points of  $\Omega_3^{\pi}$ are
\[E_1=\frac{1}{2}\left[\begin{array}{ccc}
0&1&1\\
1&0&1\\
1&1&0\end{array}\right]\mbox{ and }
E_2=\frac{1}{2}\left[\begin{array}{ccc}
1&1&0\\
1&0&1\\
0&1&1\end{array}\right].\]
In fact, we have
$A=\frac{1}{2}(E_2+I_3)$.
\hfill{$\Box$}
}
\end{example}

The follow corollary is also in \cite{AC}.

\begin{corollary}\label{cor:centroext} A $2m\times 2m$ centrosymmetric nonnegative integral matrix $A$ with all row and column sums equal to $k$ can be expressed as a sum of $k$ centrosymmetric permutation matrices.
\end{corollary}

\begin{proof}
By Theorem \ref{th:centroext} the doubly stochastic matrix $(1/k)(A)$ can be written as a convex combination of centrosymmetric permutation matrices. This implies that there exists a centrosymmetric permutation matrix $P$ whose 1's correspond to nonzero positions in $A$. Now use induction on $A-P$.\end{proof}

In \cite{CN} there is also an investigation of the convex polytope $\Omega_n^{\pm}$, defined to be the convex hull of the $n\times n$ signed permutation matrices (so $n$ is even). Under somewhat more general considerations, it is proved that $\dim \Omega_n^{\pm}=n^2$, and that the set ${\mathcal E}(\Omega_n^{\pm})$ of extreme points of $\Omega_n^{\pm}$ is precisely the set ${\mathcal P}_n^{\pm}$ of $n\times n$ signed permutation matrices.
This second assertion  is equivalent to the assertion in Theorem \ref{th:centroext} that when $n$ is even, the set of extreme points of $\Omega_n^{\pi}$ is the set of $n\times n$ centrosymmetric permutation matrices

We now turn to the dimension of the polytope $\Omega_n^{\pi}$.

\begin{example}\label{ex:dim4} {\rm Let $n=4$. 
It is straightforward to check that there are exactly eight centrosymmetric $4\times 4$ permutation matrices and that the following six matrices $A_1,A_2,A_3,A_4,A_5,A_6$ are linearly independent:
\[A_1=
\left[\begin{array}{c|c|c|c}
1&&&\\ \hline &1&&\\ \hline &&1&\\ \hline &&&1\\ \end{array}\right],
A_2=\left[\begin{array}{c|c|c|c}
&1&&\\ \hline 1&&&\\ \hline &&&1\\ \hline &&1&\\ \end{array}\right],
A_3=\left[\begin{array}{c|c|c|c}
&&&1\\ \hline &1&&\\ \hline &&1&\\ \hline 1&&&\\ \end{array}\right],\]
\[A_4=\left[\begin{array}{c|c|c|c}
1&&&\\ \hline &&1&\\ \hline &1&&\\ \hline &&&1\\ \end{array}\right],
A_5=\left[\begin{array}{c|c|c|c}
&&1&\\ \hline 1&&&\\ \hline &&&1\\ \hline &1&&\\ \end{array}\right],
A_6=\left[\begin{array}{c|c|c|c}
&1&&\\ \hline &&&1\\ \hline 1&&&\\ \hline &&1&\\ \end{array}\right].\]
The other two centrosymmetric $4\times 4$ permutation matrices, namely,
\[A_7=\left[\begin{array}{c|c|c|c}
&&&1\\ \hline &&1&\\ \hline &1&&\\ \hline 1&&&\\ \end{array}\right],
A_8=\left[\begin{array}{c|c|c|c}
&&1&\\ \hline &&&1\\ \hline1 &&&\\ \hline &1&&\\ \end{array}\right]\]
are linearly combinations of $A_1,A_2,A_3,A_4,A_5,A_6$. It can  be checked that $A_1,A_2,A_3,A_5,A_7,A_8$ are also linearly independent.
We conclude that the affine dimension of $\Omega_4^{\pi}$ is 6 and thus $\dim \Omega_4^{\pi}=5$.\hfill{$\Box$}
}\end{example}

\begin{theorem}\label{th:csdimeven}
Let $n$ be an even integer. Then
\[\dim \Omega_n^{\pi}=\frac{(n-1)^2+1}{2}.\]
\end{theorem}

\begin{proof} Let 
\[A=\left[\begin{array}{cc}
A_1&A_2\\
A_2^{\pi}&A_1^{\pi}\end{array}\right]\]
be an $n\times n$ centrosymmetric, doubly stochastic
matrix $A$ where $A_1$ is $\frac{n}{2}\times \frac{n}{2}$. Since $A$ is doubly stochastic, all of the entries of $A$ are determined once one knows the $(\frac{n}{2})^2$ entries in $A_1$ and the $(\frac{n}{2}-1)^2$ entries in the leading $(\frac{n}{2}-1)\times (\frac{n}{2}-1)$ submatrix of $A_2$.
(More generally, if $T$ is any spanning tree of the complete bipartite graph $K_{\frac{n}{2},\frac{n}{2}}$, then all of the entries of $A$ are determined once the $(\frac{n}{2})^2$ entries in $A_1$ and those entries of $A_2$ not corresponding to the edges of $T$ are specified.) Thus the dimension of $\Omega_n^{\pi}$ does not exceed that given in the theorem. 

Let $A_2^*$ be obtained from $A_2$ by reversing the order   of its columns.
Then the  matrix $\frac{1}{2}(A_1+A_2^*)$ is a matrix $B$ in $\Omega_{\frac{n}{2}}$ and every such $B\in \Omega_{\frac{n}{2}}$  can be obtained in this way. Since $\dim \Omega_{\frac{n}{2}}=(\frac{n}{2}-1)^2$, there is a small ball of dimension $(\frac{n}{2}-1)^2$ in the interior of  $ \Omega_{\frac{n}{2}}$.  Choosing any  $B$ in this ball, $A_1$ and $A_2$ can be arbitrarily chosen subject to  $A_1$ and $A_2$ having nonnegative entries and $\frac{1}{2}(A_1+A_2^*)=B$. We can choose the $\left(\frac{n}{2}\right)^2$ entries of $A_1$ independently provided they are small enough and then define $A_2^*$ by $A_2^*=B-A_1$.
Thus
\begin{eqnarray*}
\dim \Omega_n^{\pi}&\ge &\dim \Omega_{\frac{n}{2}} +\left(\frac{n}{2}\right)^2\\
&\ge &\left(\frac{n}{2}-1\right)^2+\left(\frac{n}{2}\right)^2\\
&=&
\frac{(n-1)^2+1}{2}\end{eqnarray*}
and hence we have equality.
\end{proof}

Let ${\mathcal C}_n=\Re ({\mathcal P}_n)$ be the real linear space spanned by the $n\times n$ centrosymmetric permutation matrices.
It follows from Theorems \ref{th:centroext} and \ref{th:csdimeven} that for $n$ even, $\dim {\mathcal C}_n=\frac{(n-1)^2+1}{2}+1$ and hence there
exists a set of $\frac{(n-1)^2+1}{2}+1$ linearly independent $n\times n$ centrosymmetric permutation matrices which thus form a basis of ${\mathcal C}_n$.  We now obtain  a basis of ${\mathcal C}_n$. 
 
 Let ${\mathcal L}_m={\Re}({\mathcal P}_m) $ be the real  {\it permutation linear space}  spanned by the $m\times m$ permutation matrices.  Then  $\dim {\mathcal L}_m =(m-1)^2+1$, and there are several known bases of ${\mathcal L}_m$
 (see e.g. \cite{BMe}). 
 
\begin{theorem}\label{th:csperbasis}
Let $n=2m$ be an even integer. Let ${\mathcal B}_m$ be any basis 
of ${\mathcal L}_m$. Then a basis of the linear space ${\mathcal C}_n=\Re ({\mathcal P}_n)$ is obtained as follows:
\begin{itemize}
\item[\rm (i)] The $(m-1)^2+1$ centrosymmetric permutation matrices of the form 
\[\left[\begin{array}{c|c}
P&O_m\\ \hline
O_m&P_m^{\pi}\end{array}\right]\quad (P\in {\mathcal B}_m),\]
\item[\rm (ii)] The  $m^2$ centrosymmetric permutation matrices of the form
\[\left[\begin{array}{c|c}
Q_{i,j}&P_{i,j}\\ \hline
P_{i,j}^{\pi}&Q_{i,j}^{\pi}\end{array}\right]\quad (1\le i,j\le m),\]
where $P_{i,j}$ is the $m\times m$ $(0,1)$-matrix with exactly one $1$
 and this $1$ is in its position $(i,j)$, and $Q_{i,j}$ is any  $m\times m$ $(0,1)$-matrix with all $0$'s in row $i$ and column $m+1-j$ such that deleting its row $i$ and column $m+1-j$ results in a permutation matrix.
 \end{itemize}
 \end{theorem}
 
 \begin{proof} The number of matrices specified in the theorem equals
 $(m-1)^2+1+m^2=2m^2-2m+2$ and this equals $\frac{(n-1)^2+1}{2}+1$
 for  $n=2m$. The set of matrices specified in (i) and (ii)  clearly constitute a set of linearly independent centrosymmetric 
permutation  matrices of cardinality $2m^2-2m+1=\dim {\mathcal C}_n$
 \end{proof}

\begin{example}{\rm \label{ex:n=4}
Let $n=4$. Then $\dim {\mathcal C}=6$, and $\dim {\mathcal L}_2=2$ where $\{I_2,L_2\}$ is a basis of ${\mathcal L}_2$. By Theorem \ref{th:csperbasis}, the following matrices form a basis of ${\mathcal C}_4$:
\[\left[\begin{array}{c|c||c|c}
1&&&\\ \hline
&1&&\\ \hline\hline
&&1&\\ \hline
&&&1\end{array}\right],
\left[\begin{array}{c|c||c|c}
&1&&\\ \hline
1&&&\\ \hline\hline
&&&1\\ \hline
&&1&\end{array}\right],\]
\[\left[\begin{array}{c|c||c|c}
&&1&\\ \hline
1&&&\\ \hline\hline
&&&1\\ \hline
&1&&\end{array}\right],
\left[\begin{array}{c|c||c|c}
&&&1\\ \hline
&1&&\\ \hline\hline
&&1&\\ \hline
1&&&\end{array}\right],
\left[\begin{array}{c|c||c|c}
1&&&\\ \hline
&&1&\\ \hline\hline
&1&&\\ \hline
&&&1\end{array}\right],
\left[\begin{array}{c|c||c|c}
&1&&\\ \hline
&&&1\\ \hline\hline
1&&&\\ \hline
&&1&\end{array}\right].\]
} \hfill{$\Box$}

\end{example}

We now consider $\dim \Omega_n^{\pi}$ when $n$ is odd.

\begin{theorem}\label{th:csdimodd}
Let $n$ be an odd  integer. Then
\[\dim \Omega_n^{\pi}=\frac{(n-1)^2}{2}.\]
\end{theorem}

\begin{proof} Let $n=2m+1$ so that the assertion in the theorem is that $\dim \Omega_n^{\pi}=2m^2$.  Let
\begin{equation}\label{eq:small}A=\left[\begin{array}{c|c|c}
A_1&u&A_2\\ \hline
v&x&v^{\pi}\\ \hline
A_2^{\pi}&u^{\pi}&A_1^{\pi}\end{array}\right]\end{equation}
be an $n\times n$ centrosymmetric, doubly stochastic
matrix $A$ where $A_1$ and $A_2$ are $m\times m$ matrices, and
$u$ is $m\times 1$,  $v$ is  $1\times m$, and $x$ is $1\times 1$.
Since $A$ is doubly stochastic, the entries of $u$, $v$, and  $x$ are uniquely determined 
by the entries of $A_1$ and $A_2$. Thus the dimension of $\Omega_n^{\pi}$ is at most $2m^2$. We now show the reverse inequality holds.

We can choose  $m^2$ nonnegative entries for $A_1$ and $m^2$ nonnegative  entries for $A_2$ independently, but small enough, so that their row sum vectors $R_1=(r_1,r_2,\ldots,r_m)$
 and $R_2=(r_1',r_2',\ldots,r_m')$ and
 column sum vector $S_1=(s_1,s_2,\ldots,s_m)$, and  $S_2=(s_1',s_2',\ldots,s_m')$ satisfy
 \begin{itemize}
  \item[\rm (i)] $\sum_{i=1}^mr_i=\sum_{i=1}^m s_i$ and $\sum_{i=1}^mr_i'=\sum_{i=1}^m s_i'$,
 \item[\rm (ii)] $r_i+r_i'\le 1$ and $s_i+s_i'\le 1$, $(1\le i\le m)$,
 \item[\rm (iii)]  $\sum_{i=1}^m(r_i+r_i')\ge m-\frac{1}{2}$,
 \item[\rm (iv)] $ \sum_{i=1}^m(s_i+s_i')\ge m-\frac{1}{2}$.
 \end{itemize}
 Then defining $u_i=1-(r_i+r_i')$, $v_i=1-(s_i+s_i')$, $(1\le i\le m)$, and
 $x=1-2\sum_{i=1}^mu_i$, we  obtain a matrix  (\ref{eq:small}) in
 $\Omega_n^{\pi}$.
 We now conclude that
\[\dim \Omega_n^{\pi}\ge m^2+m^2=2m^2= \frac{(n-1)^2}{2}.\]
\end{proof}
By Theorem \ref{th:centroext}, the extreme points of $\Omega_n^{\pi}$ when $n$ is even are the $n\times n$ centrosymmetric permutation matrices.
We next characterize the extreme points of  $\Omega_n^{\pi}$ when $n$ is odd. 

Let $X$ be an $n\times n$ matrix.  If $\emptyset\ne K,L\subseteq \{1,2,\ldots,n\}$, then $X[K,L]$ denotes the $|K|\times |L|$ submatrix of $X$ determined by rows with index in $K$ and  columns with index  in $L$; $X(K,L)$ denotes the complementary  submatrix $X[\overline{K},\overline{L}]$. Now let $n=2m+1$ be odd, and let $|K|=|L|=r$. If we have that
\begin{itemize}
\item[\rm (a)] $m+1\in K\cap L$, 
\item[\rm (b)] $\{i:i\in K\}=\{n+1-i:i\in K\}$, and 
\item[\rm (c)]$\{j\in L\}=\{n+1-j:j\in L\}$, 
\end{itemize}
then we call $X[K,L]$ an $r\times r$   {\it centro-submatrix} of $X$. A centro-submatrix of $X$ is $r\times r$ for some odd integer $r$. Note that if $X[K,L]$ is a centro-submatrix of $X$, then $X^{\pi}[K,L]=X[K,L]^{\pi}$. If $X$ is centrosymmetric, then so is any centro-submatrix of $X$. If the submatrices $X[K,\overline{L}]$ and $X[\overline{K},L]$ are zero matrices, then we say that $X[K,L]$ is {\it isolated}. 

In what follows, it is convenient to refer to the {\it bipartite multigraph} G$(X)$
associated with a $k\times k$ nonnegative integral matrix $X=[x_{ij}]$. This graph has $2m$ vertices $\{u_1,u_2,\ldots,u_k\}\cup\{v_1,v_2,\ldots,v_k\}$ (corresponding to the rows and columns, respectively) with $x_{ij}$ edges between $u_i$ and $v_j$ $(1\le i,j\le k)$.
Let $r=2l-1$ be an odd integer and let $P=[p_{ij}]$ be an $r\times r$ permutation matrix. We say that $P$ is an {\it anti-centrosymmetric permutation matrix}  provided that $p_{ij}=1$ implies that $p_{r+1-i,r+1-j}=0$, in particular, $p_{ll}=0$. 
If $P$ is anti-centrosymmetric, then $P+P^{\pi}$ is a centrosymmetric $(0,1)$-matrix with exactly two 1's in each row and column. We say that $P+P^{\pi}$ is {\it cyclic}
provided that its associated bipartite graph is a cycle of length $2r$, that is, can be permuted to the form illustrated in (\ref{eq:bipcycle})  for $n=5$:
\begin{equation}\label{eq:bipcycle}
\left[\begin{array}{c|c|c|c|c}
1&1&&&\\ \hline
&1&1&&\\ \hline
&&1&1&\\ \hline
&&&1&1\\ \hline
1&&&&1\end{array}\right].\end{equation}

If $n$ is odd, the characterization of the extreme points of $\Omega_n^{\pi}$ involves anti-centrosymmetric, cyclic permutation matrices.

\begin{theorem}\label{th:csextreme}
Let $n=2m-1\ge 3$ be an odd integer. Then 
${\mathcal E}(\Omega_n^{\pi})$ is the set of $n\times n$ matrices $A=[a_{ij}]$ for which there exists a permutation matrix $Q=[q_{ij}]$ such that $A=\frac{1}{2}(Q+Q^{\pi})$ and one of the following holds:
\begin{itemize}
\item[\rm (a)] $q_{mm}=1$ and $Q$ is a centrosymmetric permutation matrix $($thus $Q=Q^{\pi}$  and $A$ is a centrosymmetric permutation matrix$)$,
or
\item[\rm (b)] $q_{mm}=0$ and there  exists an odd integer $r$, an
$r\times r$  anti-centrosymmetric permutation matrix $P$ such that $P+P^{\pi}$ is cyclic, and   an isolated   centro-submatrix $A[K,L]$  of $A$ with $r=|K|=|L|$  and
$A[K,L]=\frac{1}{2}(P+P^{\pi})$ such that
the  complementary submatrix $A[{\overline{K}}, {\overline{L}}]$ is a centrosymmetric permutation matrix $($thus $Q[\overline{K},\overline{L}]=Q[\overline{K},\overline{L}]^{\pi})$.
\end{itemize}
\end{theorem}

Before giving the proof, we exhibit an example of an extreme point of type (b).

\begin{example}\label{ex:extreme}{\rm
First we note that in (b), the matrix $A[K,L]$ can be permuted to the  form given in $(\ref{eq:C_k})$ .
Let 
\[Q=\left[\begin{array}{c|c|c|c|c}
&1&&&\\ \hline
1&&&&\\ \hline
&&&1&\\ \hline
&&&&1\\ \hline
&&1&&\end{array}\right],\mbox{ where }
A=\frac{1}{2}(Q+Q^{\pi})=\frac{1}{2}\left[\begin{array}{c|c|c|c|c}
&1&1&&\\ \hline
2&&&&\\ \hline
&1&&1&\\ \hline
&&&&2\\ \hline
&&1&1&\end{array}\right].
\]
Let
\[P=\left[\begin{array}{ccc}
0&1&0\\1&0&0\\ 0&0&1\end{array}\right],\] an anti-centrosymmetric permutation matrix,
and let
$K=\{1,3,5\}$ and $L=\{2,3,4\}$.  Then  $A[K,L]$ is isolated with 
\[A[K,L]=\frac{1}{2}(P+P^{\pi})=\frac{1}{2}\left[\begin{array}{ccc}
1&1&0\\
1&0&1\\
0&1&1\end{array}\right]\]
where $(P+P^{\pi})$ is cyclic. We also have
\[A[\overline{K},\overline{L}]=\left[\begin{array}{cc} 1&0\\0&1\end{array}\right],\]
a centrosymmetric permutation matrix.
}\hfill{$\Box$}
\end{example}

\begin{proof} (of Theorem \ref{th:csextreme})
Let $A=[a_{ij}]\in\Omega_n^{\pi}$. Since $A$ is doubly stochastic, $A$ is a convex combination of  permutation matrices 
\[A=\sum_{i=1}^r c_iP_i.\]
We have
\[A=A^{\pi}=\sum_{i=1}^r c_iP_i^{\pi},\] and hence
\begin{equation} \label{eq:centrosum}
A=\sum_{i=1}^r c_i\left(\frac{1}{2}(P_i+P_i^{\pi})\right).\end{equation}
Thus the extreme points of $\Omega_n^{\pi}$ are all of the form
$\frac{1}{2}(P+P^{\pi})$ for some $n\times n$ permutation matrix $P$.
If $a_{mm}=1$, then it follows from Theorem \ref{th:centroext} that $A$ is an extreme point of $\Omega_n^{\pi}$ if and only if deleting row and column $m$ gives a centrosymmetric permutation matrix.

Now assume that $a_{mm}\ne 1$. By (\ref{eq:centrosum}), if $A$ is an extreme point of $\Omega_n^{\pi}$, then its entries are contained in $\{0,\frac{1}{2},1\}$.  From centrosymmetry we see that $\sum_{i=1}^{2m-1} a_{mi}\ge 2(\frac{1}{2})+a_{mm}=1+a_{mm}$,  and it follows that $a_{mm}=0$.

 A centrosymmetric doubly stochastic matrix with the structure given in (b) is an extreme point of $\Omega_n^{\pi}$ since it is easily checked that its zeros determine a unique matrix in $\Omega_n^{\pi}$. Our remaining task is to prove that every $n\times n$ matrix $A=[a_{ij}]=\frac{1}{2}(Q+Q^{\pi})$ where $Q$  is a permutation matrix and $A$ is an  extreme point of $\Omega_n^{\pi}$ with $a_{mm}=0$ satisfies (b).
 Assume we have such an $A$. Then $B=[b_{ij}]=Q+Q^{\pi}$ is a $(0,1,2)$-matrix  each of whose row and column sums equals 2.  Thus the bipartite multigraph BG$(B)$ consists of a pairwise disjoint collection of cycles and double edges (which we regard as cycles of length 2).
Consider a cycle corresponding to row indices $K$ and column indices $L$, where we must have $|K|=|L|$. Since $B$ is centrosymmetric, there is also a cycle corresponding to row indices $K'=\{n+1-i: i\in K\}$ and $L'=\{n+1-j:j\in L\}$, and one of the following holds:
\begin{itemize}
\item[\rm (a)] $K=K'$ and $L=L'$, and thus $B[K,L]$ is an isolated and centrosymmetric, centro-submatrix of $B$ with all row and column sums equal to 2, or
\item[\rm (b)] $K\cap K'=\emptyset$ and $L\cap L'=\emptyset$ and thus $B[K\cup K',L\cup L']$ is an isolated  and centrosymmetric centro-submatrix of $B$.
\end{itemize}
 It follows that there exist
 two partitions $K_1,K_2,\ldots,K_s$ and $L_1,L_2,\ldots,L_s$ 
 of $\{1,2,\ldots,n\}$ such that 
 $B[K_i,L_i]$  is an isolated and centrosymmetric centro-submatrix of $B$ for each $i=1,2,\ldots,s$. The bipartite multigraphs  
 BG$(B[K_i,L_i])$ are  either  cycles, or consist of two disjoint cycles which are 180 degrees rotations of one another.
 
 Exactly one of the pairs $K_i,L_i$, say $K_1,L_1$, satisfies  $m\in K_1\cap L_1$, and then $|K_1|=|L_1|$ is  odd and the bipartite multigraph BG$(B[K_1,L_1])$ is one cycle of length $2r$ for some odd integer $r$.  Suppose there exists a $p\ne 1$
 such that $|K_p|=|L_p|>2$. Then  BG$(B[K_p,L_p])$ consist of two disjoint cycles which are 180 degrees rotations of one another. Thus
 \[B[K_p,L_p]=S+S^{\pi}\]
 where $S$ and $S^{\pi}$ are centrosymmetric permutation matrices. It then follows that the matrices $C_1$ and $C_2$ obtained from $B=Q+Q^{\pi}$ by replacing
 $B[K_p,L_p]$ with $S$ and $S^{\pi}$, respectively, are centrosymmetric permutation matrices and $B=S+S^{\pi}$.  Therefore, $A=\frac{1}{2}(S+S^{\pi})$ implying that $A$ is not an extreme point of $\Omega_n^{\pi}$.
 We conclude  that $B[\overline{K_1},\overline{L_1}]=2U$ where $U$ is a centrosymmetric permutation matrix and $B[K_1,L_1]$ is an isolated, centrosymmetric centro-submatrix of $B$. Therefore $A$ has the properties given in (b).
\end{proof}

Since, in the case of $n$ odd, there are extreme points of $\Omega_n^{\pi}$ that are not centrosymmetric permutation matrices, the maximum number of linearly independent centrosymmetric $n\times n$ permutation matrices may be less than $\dim \Omega_n^{\pi}$. In fact, we have the following. Let $n=2m+1$. Then an $n\times n$ centrosymmetric permutation matrix $P=[p_{ij}]$ satisfies $p_{m+1,m+1}=1$, from which it follows that the maximum number of linearly independent $(2m+1)\times (2m+1)$ centrosymmetric permutation matrices equals the maximum number of linearly independent $2m\times 2m$ centrosymmetric permutation matrices and thus equals $((2m-1)^2+1)/2+1$. Notice that this number is $2m$ less than $\dim \Omega_{2m}^{\pi}$.

We conclude this section with some comments related to latin squares. Recall that an $n\times n$ latin square is a matrix
$T=[t_{ij}]$ with entries from $\{1,2,\ldots,n\}$ with no repeated integer in a row or column. Thus
$T=1P_1+2P_2+\cdots+nP_n$ for some permutation matrices $P_1,P_2,\ldots,P_n$ satisfying $P_1+P_2+\cdots+P_n=J_n$ where $J_n$ is the $n\times n$ matrix of all 1's. A centrosymmetric latin square is one where
$T^{\pi}=T$ and so necessarily $n$ is even, say $n=2m$. It is easy to construct $2m\times 2m$ centrosymmetric latin squares from $m\times m$ latin squares $U_m$:
\[\left[\begin{array}{c|c}
U_m&U_m+mJ_m\\ \hline
(U_m+mJ_m)^{\pi}&U_m^{\pi}\end{array}\right].\]
Thus, in this construction, the integers $\{1,2,\ldots,m\}$ are isolated from the integers $\{m+1,m+2,\ldots,2m\}$. In general, any decomposition of $J_{2m}=Q_1+Q_2+\cdots+Q_{2m}$ as a sum of $2m$ centrosymmetric permutation matrices $Q_1,Q_2,\ldots,Q_{2m}$ gives a centrosymmetric latin square $1Q_1+2Q_2+\cdots+2mQ_{2m}$.

In the next section we consider the convex polytope of $n\times n$ doubly stochastic matrices which are both symmetric and Hankel-symmetric, and thus are centrosymmetric as well.

\section{Symmetric \& Hankel-Symmetric Doubly Stochastic Matrices}

Let $\Omega_n^{t\&h}$ be the set of $n\times n$ doubly stochastic matrices $A$ that are both symmetric  and Hankel-symmetric\footnote{Sometimes also called {\it doubly symmetric}.} ($A^t=A$ and $A^h=A$) and so centrosymmetric ($A^{\pi}=A$) as well. Then $\Omega_n^{t\&h}$ is a convex polytope and is a subpolytope of $\Omega_n^{\pi}$.
Let ${\mathcal P}_n^{t\&h}$  be the set of $n\times n$ symmetric and Hankel-symmetric permutation matrices. Each matrix in  ${\mathcal P}_n^{t\&h}$  is clearly an extreme point of  $\Omega_n^{t\&h}$ but in general there are other extreme points.
The number of  $n\times n$ symmetric and Hankel-symmetric permutation matrices is known \cite{BBS}. This is because to say that a permutation matrix $P$ is symmetric means that it is an involution ($P^2=I_n$). In \cite{BM} the elementary observation is made   that any two of the three properties
of symmetric, Hankel-symmetric, and centrosymmetric implies the third. In \cite{BBS} it is shown that the
number of centrosymmetric involutions of order $n=2k$ equals
\begin{equation}\label{eq:number}
\sum_{h=0}^{\lfloor \frac{k}{2}\rfloor}\frac{(2k)!!}{(k-2h)!h!2^{2h}},\end{equation}
where the {\it double factorial} $(2k)!!$ equals $(2k)(2k-2)\cdots (2)$. If $n=2k+1$ is odd, then in a centrosymmetric permutation, $(k+1)$ is a fixed point so that (\ref{eq:number}) also gives the 
number of symmetric and Hankel-symmetric permutations of order $n=2k+1$.

\begin{example}\label{ex:sym&hankel} {\rm The following is a symmetric and Hankel-symmetric doubly stochastic matrix, indeed a permutation matrix:
\[A_1=\left[\begin{array}{c|c|c|c|c}
&&&1&\\ \hline
&&&&1\\ \hline
&&1&&\\ \hline
1&&&&\\ \hline
&1&&\end{array}\right].\]
Its associated loopy graph $G(A_1)$ consists of vertices $1,2,3,4,5$ and edges $\{1,4\}, \{2,5\}$, and $\{3\}$ (a loop). We can also consider its {\it Hankel loopy graph} $G^h(A_1)$, the loopy graph  with respect to the Hankel diagonal.\footnote{Note that the graphs $G(A)$ and $G^h(A)$ are not 
considered as multigraphs; at most one edge (corresponding to a nonzero entry) joins a pair of vertices. This is in contrast to the bipartite multigraph BG$(A)$ in which more than one edge can join a pair of vertices, equivalently,  edges have weights assigned to them.}  This graph has the same vertex set $1,2,3,4,5$ but associated with rows labeled $1,2,3,4,5$ and columns labeled
$5,4,3,2,1$ in these orders. The edges of  $G^h(A)$ are $\{1,2\}$, $\{4,5\}$, and $\{3\}$.

Now consider the matrix in $\Omega_6^{t\&h}$ given by
\[A_2=\frac{1}{2}\left[\begin{array}{c|c|c|c|c|c}
&1&&&&1\\ \hline
1&&1&&&\\ \hline
&1&&1&&\\ \hline
&&1&&1&\\ \hline
&&&1&&1\\ \hline
1&&&&1&\end{array}\right].\]
Then  $G(A_2)$ is the cycle 1---2---3---4---5---6---1 of length 6 while $G^h(A_2)$ consists of two paths 1---5---3  and 6---2---4 with loops at vertices 1, 3, 4, and 6.
An easy check shows that any matrix in $\Omega_6^{t\&h}$ which has zeros 
in all the positions that $A$ has zeros must equal $A$. (Thus the matrix $A_2$ is uniquely determined in $\Omega_6^{t\&h}$ by its graph. From this it follows that $A$ is an extreme point of $\Omega_6^{t\&h}$.
\hfill{$\Box$} } 
\end{example}

We now turn to the dimension of the polytope $\Omega_n^{t\&h}$. An $n\times n$ matrix is both symmetric and Hankel-symmetric if and only if it is of the form
\begin{equation}\label{eq:hsmall}A=\left[\begin{array}{c|c}
A_1&A_2\\ \hline
A_2^{\pi}&A_1^{\pi}\end{array}\right]\mbox{ ($n=2m$ even), and } 
A=\left[\begin{array}{c|c|c}
A_1&u&A_2\\ \hline
u^t&x&{u^t}^\pi\\ \hline
A_2^{\pi}&u^{\pi}&A_1^{\pi}\end{array}\right]\mbox{ ($n=2m+1$ odd),}
\end{equation}
where $A_1$ is symmetric and $A_2$ is Hankel-symmetric.
First suppose that  $n=2m$ is even.
   Let $Q_1$ be an $m\times m$ permutation matrix and let $Q$ be the  centrosymmetric permutation matrix $Q_1\oplus Q_1^{\pi}$. Then
$QAQ^t$  is also a symmetric and Hankel-symmetric matrix obtained by reordering the first $m$ rows and columns of $A$ in the same way according to the permutation defined by $Q$ and reordering the last $m$ rows and columns in the same way according to the permutation defined by $Q^{\pi}$. 
If $n=2m+1$ is odd, then an $n\times n$ centrosymmetric permutation matrix has a 1 in position $(m+1,m+1)$ (the corresponding permutation fixes $m+1$) and the above observations apply in the obvious way.

\begin{theorem}\label{th:dims&t}
Let $n$ be a positive integer. Then
\[\dim \Omega_n^{t\&h}=\left\{\begin{array}{ll}
\frac{n^2}{4} &\mbox{ if $n$ is even,}\\
\frac{n^2-1}{4}& \mbox{ if $n$ is odd.}\end{array}\right.\]
\end{theorem}

\begin{proof} We refer to (\ref{eq:hsmall}). First assume that $n=2m$ is even so that the assertion is that $\dim \Omega_n^{t\&h}=m^2$.
Since $A_1$ is symmetric, $A_1$ is determined by its ${m+1}\choose {2}$ entries on and above its main diagonal (a triangular matrix $A_1'$ with zeros below the main diagonal), and $A_2$ is then determined by its 
${m}\choose {2}$  entries below its Hankel diagonal (a triangular matrix $A_2'$ with zeros on and above its Hankel diagonal) and the property that the row sums of $A$ equal 1. Thus $\dim \Omega_n^{t\&h}\le m^2$. Moreover, $A_1'$
and $A_2'$ can be combined in the obvious way to yield an $m\times m$ matrix determined by $A_1'$ on and above the main diagonal and by $A_2'$ below the main diagonal.  Let $X'$ be an $m\times m$ matrix all of whose entries are nonzero with row and column sums all equal to zero. We may reverse the above construction and use $X'$ to construct a matrix $X$ of the form of $A$ which is symmetric and Hankel-symmetric and has all row and column sums equal to zero. Then
$J_n\pm \epsilon X$ is in $\Omega_n^{t\&h}$ for all $\epsilon$ small enough, establishing that $\dim \Omega_n^{t\&h}\ge m^2$ as well.

Now assume that $n=2m+1$ is odd so that the assertion is that $\dim \Omega_n^{t\&h}=m(m+1)$. We can proceed in a similar way as in the even case. The matrix $A_1$ is determined by its ${m+1}\choose {2}$ entries on and above its main diagonal, and $A_2$ is determined by its 
${m+1}\choose {2}$  entries on and above its Hankel diagonal. The entries of $u$ are now determined by the property that the row sums of $A$ equal 1. Thus
$\dim \Omega_n^{t\&h}\le m(m+1)$. Thus the dimension of $ \Omega_n^{t\&h}$ does not exceed that given in the theorem. Let $X'$ and $D'$ be $m\times m$ matrices with $D$ a Hankel-diagonal matrix, such that  all entries of $X'$ are nonzero, all Hankel-diagonal entries of $D'$ nonzero, and   with row and column sums all equal to zero. We may use $X'$ and $D'$ to construct a matrix $X+D$ of the form of $A$ which is symmetric and Hankel-symmetric and has all row and column sums equal to zero. Then
$J_n\pm \epsilon (X+D)$ is in $\Omega_n^{t\&h}\le m^2$ for all $\epsilon$ small enough, establishing that $\dim \Omega_n^{t\&h}\ge m(m+1)$ as well.
\end{proof}

It follows from Theorem \ref{th:dims&t} that there are at most
$\frac{n^2}{4}+1$ linearly independent  symmetric and Hankel-symmetric $n\times n$ permutation matrices if $n$ is even, and at most
$\frac{n^2-1}{4}+1$ if $n$ is odd. We now construct  a set of
linearly independent  symmetric and Hankel-symmetric $n\times n$ permutation matrices of maximum size, that is, a basis for the real linear space ${\Re } ({\mathcal P}_n^{t\& h})$ spanned by the symmetric and Hankel-symmetric $n\times n$ permutation matrices.

Let $n=2m$. For $1\le i<j\le m$, we denote by $T_m^{i,j}$ the $m\times m$ permutation matrix corresponding to the transposition $(i,j)$.
Similary,  $H_m^{i,j}$ denotes the $m\times m$ permutation matrix corresponding to the {\it Hankel transposition} $(i,j)$ taken with respect to the Hankel diagonal; reversing the columns of $H_m^{i,j}$  gives the matrix $T_m^{i,j}$. Let
\[\tilde{T}_n^{i,j}=\left[\begin{array}{cc} T_m^{i,j}&O_m\\ O_m&{T_m^{i,j}}^{\pi}\end{array}\right]
\mbox{ and }
\tilde{H}_n^{i,j}=\left[\begin{array}{cc} O_m&H_m^{i,j}\\ {H_m^{i,j}}^{\pi}&O_m\end{array}\right]\quad (1\le i<j\le m).\]
For $0\le i\le m$, let $A_i$ be the $m\times m$ $(0,1)$-matrix with 1's only in the first $i$ positions of the main diagonal, and let $B_i$ be the $m\times m$ $(0,1)$-matrix with 1's only in the last $m-i$ positions of the Hankel diagonal. Finally, let
\[U_i=\left[\begin{array}{cc} A_i&B_i\\ B_i^{\pi}&A_i^{\pi}\end{array}\right] \quad (0\le i\le m).\]
Note that $U_m=I_m$ and $U_0=H_m$. Let
\[{\mathcal B}_n=\{\tilde{T}_n^{i,j}: 1\le i<j\le m\}\cup 
\{\tilde{H}_n^{i,j}: 1\le i<j\le m\}\cup\{U_i:0\le i\le m\},\]
a set of $n\times n$ symmetric and Hankel-symmetric permutation matrices.
Then
\[|{\mathcal B}_n|=2{m\choose 2}+m+1=m^2+1=\frac{n^2}{4}+1.\]

\begin{theorem}\label{th:s&hsindepeven} Let $n$ be an even  positive integer. Then the set  ${\mathcal B}_n$ forms a basis of the vector space ${\Re}({\mathcal P}_n^{t\&h})$ spanned by the symmetric and Hankel-symmetric 
$n\times n$ permutation matrices.
\end{theorem}

\begin{proof} Since $|{\mathcal B}_n|=\dim {\Re }({\mathcal P}_n^{t\&h})$, we only need establish that ${\Re}({\mathcal P}_n^{t\&h})$ is linearly independent. Let $2m\times 2m$ permutation  matrices be defined by
$\{\tilde{T}_n^{i,j}: 1\le i<j\le m\}\cup 
\{\tilde{H}_n^{i,j}: 1\le i<j\le m\}$ is clearly linearly independent since each has an entry equal to 1 where the other matrices all have 0's.
Each matrix in the  set $\{U_i:0\le i\le m\}$ of matrices is clearly linearly independent has zeros in all positions not on its main or Hankel diagonal; a nontrivial linear combination of matrices in the set
$\{\tilde{T}_n^{i,j}: 1\le i<j\le m\}\cup 
\{\tilde{H}_n^{i,j}: 1\le i<j\le m\}$ cannot share this property. It follows that ${\mathcal B}_n$  is a linearly independent set and hence a basis of ${\Re}({\mathcal P}_n^{t\&h})$.
\end{proof}

Thus in the even case, there is a linearly independent set of 
$n\times n$ symmetric and Hankel-symmetric permutation matrices 
of size $\dim \Omega_n^{t\& h}+1$. This does not hold in the odd case. Let $n=2m+1$ be odd and let ${\mathcal B}^*_n$ be obtained
from ${\mathcal B}_{n-1}$ by inserting in each of its matrices as row $m+1$ and column $m+1$, respectively, a new row $(0,1)$-vector and a new column $(0,1)$-vector
each of size $2m+1$ whose only 1 is in their $m+1$ positions. Thus
$|{\mathcal B}^*_n|=|{\mathcal B}_{n-1}|$.

\begin{theorem}\label{th:s&hsindepodd} Let $n$ be an odd positive integer. Then the set  ${\mathcal B}^*_n$ forms a basis of the vector space ${\Re}({\mathcal P}^{s\&hs})$ spanned by the symmetric and Hankel-symmetric 
$n\times n$ permutation matrices. This basis has size $\dim\Omega_n^{s\&hs}-\frac{n-1}{2}+1$.
\end{theorem}

\begin{proof}
Let $n=2m+1$. By symmetry and Hankel-symmetry it follows that a permutation matrix $P$ in ${\mathcal P}_n^{s\&hs}$ must have a 1 in position $(m+1,m+1)$ with zeros  in all other positions of row $m+1$ and column $m+1$. Thus deleting row $m+1$ and column $m+1$ in $P$ leaves a symmetric and Hankel-symmetric permutation matrix in ${\mathcal P}_{n-1}^{s\&hs}$. The matrices in  ${\mathcal B}^*_n$ then give a set of linearly independent $n\times n$ symmetric and Hankel-symmetric permutation matrices of maximum size.\end{proof}

We now turn to the extreme points of the $n\times n$ symmetric and Hankel symmetric, doubly stochastic matrices.

\begin{lemma}\label{lem:genextremepoints}
The extreme points of $\Omega_n^{t\&h}$ are all of the form
\begin{equation}\label{eq:genextremepoints}
\frac{1}{4}(P+P^t+P^h+P^{\pi})=\frac{1}{2}(R+R^t)\end{equation}
for some permutation matrix $P$, where $R=\frac{1}{2}(P+P^{\pi})$ is centrosymmetric.
\end{lemma}

\begin{proof}
Let $A\in\Omega_n^{t\&h}$. Since $\Omega_n^{t\&h}\subseteq \Omega_n$, $A$ is a convex combination of $n\times n$ permutation matrices,
\[A=\sum_{i=1}^mc_iP_i,\mbox{ where $P_i\in{\mathcal P}_n$, $c_i> 0$ $(1\le i\le m)$, and $\sum_{i=1}^mc_i=1$.}\] Since $A \in\Omega_n^{t\&h}$, we have
$A=A^t=A^h=A^{\pi}$ and thus $A$ also has the convex combination representation
\[
A=\sum_{i=1}^m c{_i}\left(\frac{1}{4} (P_i+P_i^t +P_i^h+P_i^{\pi})\right)\]
where 
\[\frac{1}{4}(P_i+P_i^t +P_i^h+P_i^{\pi})=
\frac{1}{2}\left(\frac{1}{2}(P_i+P_i^{\pi})+\frac{1}{2}(P_i+P_i^{\pi})^t\right)=\frac{1}{2}(R+R^t)\]
 and $R=\frac{1}{2}(P_i+P_i^{\pi})$ is centrosymmetric. Since every matrix in  $\Omega_n^{t\&h}$ is a convex combination of matrices of the form (\ref{eq:genextremepoints}), the extreme points of $\Omega_n^{t\&h}$ are of  this form.
\end{proof}

\begin{corollary}\label{cor:quarter} Each   entry of an extreme point of ${\Omega}_n^{t\& h}$ equals one of $0,\frac{1}{4},\frac{1}{2}, \frac{3}{4},1$, and thus has at most four nonzero entries in a row or column.\hfill{$\Box$}
\end{corollary}

Not every matrix of the form (\ref{eq:genextremepoints}) is an extreme point of 
$\Omega_n^{t\&h}$.

\begin{example}\label{ex:nonextreme}{\rm
Let 
\[P=\left[\begin{array}{c|c|c|c|c|c|c|c}
&&1&&&&&\\ \hline
&1&&&&&&\\ \hline
&&&1&&&&\\ \hline
&&&&&1&&\\ \hline
&&&&1&&&\\ \hline
&&&&&&1&\\ \hline
1&&&&&&&\\ \hline
&&&&&&&1\end{array}\right],\]
 be the permutation matrix  corresponding to the permutation of $\{1,2,\ldots,8\}$ given by $(3,2,4,6,5 ,7 ,1 ,8 )\in {\mathcal S}_8$ with a  5-cycle  $1\rightarrow 3\rightarrow 4\rightarrow 6\rightarrow 7\rightarrow 1$,  and the three fixed points 2, 5, and 8.
We have
\[\frac{1}{4} (P+P^t+P^h+P^{\pi}) =
\frac{1}{4}\left[\begin{array}{c|c|c|c|c|c|c|c}
2&&1&&&&1&\\ \hline
&2&1&&&&&1\\ \hline
1&1&&1&1&&&\\ \hline
&&1&2&&1&&\\ \hline
&&1&&2&1&&\\ \hline
&&&1&1&&1&1\\ \hline
1&&&&&1&2&\\ \hline
&1&&&&1&&2\end{array}\right],\]
a matrix in $\Omega_8^{t\&h}$ which is also the convex combination $\frac{1}{2}(R+R^{\pi})$ of the two distinct matrices $R$ and $R^{\pi}$ in $\Omega_8^{t\&h}$ given by: 
\[R=\frac{1}{4}\left[\begin{array}{c|c|c|c|c|c|c|c}
2&&1&&&&1&\\ \hline
&2&1&&&&&1\\ \hline
1&1&&2&0&&&\\ \hline
&&2&2&&0&&\\ \hline
&&0&&2&2&&\\ \hline
&&&0&2&&1&1\\ \hline
1&&&&&1&2&\\ \hline
&1&&&&1&&2\end{array}\right] \mbox{ and }
R^{\pi}=\frac{1}{4}\left[\begin{array}{c|c|c|c|c|c|c|c}
2&&1&&&&1&\\ \hline
&2&1&&&&&1\\ \hline
1&1&&0&2&&&\\ \hline
&&0&2&&2&&\\ \hline
&&2&&2&0&&\\ \hline
&&&2&0&&1&1\\ \hline
1&&&&&1&2&\\ \hline
&1&&&&1&&2\end{array}\right].\]
Thus $\frac{1}{4} (P+P^t+P^h+P^{\pi})$ is not an extreme point of 
$\Omega_8^{t\&h}$.}
\hfill{$\Box$}\end{example}

To characterize the extreme points of $\Omega_n^{t\&h}$ we first consider the case where $n$ is even.  In the next lemma we show that the sum of a permutation matrix $P$ and $P^{\pi}$ can always be represented as the sum of two centrosymmetric permutation matrices. 
In the proof of this lemma, we shall make use of the {\it multidigraph} $\Gamma(A)$ for an $n\times n$  nonnegative integral matrix $A=[a_{ij}]$. This multidigraph has vertex set $\{1,2,\ldots,n\}$ with $a_{ij}$ edges from vertex $i$ to vertex $j$ ($1\le i,j\le n$). If $A$ is a $(0,1)$-matrix, then $\Gamma(A)$ is a {\it digraph}, that is, no edge is repeated.

\begin{lemma}\label{lem:evencentroperm}
Let $P=[p_{ij}]$ be an $n\times n$ permutation matrix with $n$ even. Then there exist centrosymmetric permutation matrices $Q_1$ and $Q_2$ such that $P+P^{\pi}=Q_1+Q_2$. 
\end{lemma}

\begin{proof}
If $P$ is centrosymmetric, then $P^{\pi}=P$ and we may take $Q_1=P$ and $Q_2=P^{\pi}=P$.

Now assume that $P$ is not centrosymmetric.  
  We partition the row and column indices $\{1,2,\ldots,n\}$ into the $n/2$ pairs          
$\{i,n+1-i\}$ for $1\le i\le n/2$. Each of the $2\times 2$ submatrices $(P+P^{\pi})[\{i,n+1-i\};\{j,n+1-j\}]$ of $P+P^{\pi}$, determined by rows $i$ and $n+1-i$ and columns $j$ and $n+1-j$ equals one of
\begin{equation}\label{eq:a}
b\left[\begin{array}{cc} 1&0\\ 0&1\end{array}\right] \mbox{ and }
b\left[\begin{array}{cc} 0&1\\ 1&0\end{array}\right],\end{equation}
where $b$ is 0,1, or 2. Now consider the multidigraph $\Gamma (P+P^{\pi})$ whose $n/2$ vertices are the pairs
$\{k,n+1-k\}$ for $1\le k\le n/2$ with $b$ arcs from $\{i,n+1-i\}$ to $\{j,n+1-j\}$
if $(P+P^{\pi})[\{i,n+1-i\};\{j,n+1-j\}]$ is as given in (\ref{eq:a}) (thus no arcs if $b=0$, one arc if $b=1$, and two arcs if $b=2$). The indegree and outdegree of each vertex of $\Gamma(P+P^{\pi})$ equals 2, and thus the $(n/2)\times (n/2)$ adjacency matrix $A$ of this multidigraph  is a $(0,1,2)$-matrix with all row and column sums equal to 2. Therefore $A$ is the sum of two permutation matrices, and hence the edges of $\Gamma(P+P^{\pi})$ can be partitioned into two  spanning collection of cycles. Each such spanning collection of cycles corresponds to  a centrosymmetric permutation and this gives  $Q_1$ and $Q_2$ with $P+P^{\pi}=Q_1+Q_2$.
\end{proof}

Let $Q=[q_{ij}]$ be an $n\times n$ centrosymmetric permutation matrix, and consider a permutation cycle of $Q$ of length $k$, that is, a cycle 
$ \gamma: i_1\rightarrow i_2\rightarrow\cdots\rightarrow  i_k\rightarrow i_1$
in the digraph $\Gamma(Q)$ of $Q$. (It is convenient to identify a permutation cycle with the corresponding digraph cycle.)  Since $Q$ is centrosymmetric,  then $\gamma^{\pi}: n+1-i_1\rightarrow n+1-i_2\rightarrow \cdots\rightarrow n+1-i_k\rightarrow n+1-i_1$ is also a cycle  in the digraph $\Gamma(Q)$. We say that $\gamma$ is a {\it centrosymmetric cycle of length $k$} of $Q$ provided $\gamma^{\pi}=\gamma$.
Thus a permutation cycle of $Q$ is {\it centrosymmetric} if it is invariant under a rotation of 180 degrees. If $\gamma$ and $\gamma^{\pi}$ are disjoint, that is, $\{i_1,i_2,\ldots,i_k\}\cap\{n+1-i_1,n+1-i_2,\ldots,n+1-i_k\}=\emptyset$, then $\{\gamma,\gamma^{\pi}\}$ is called a {\it centrosymmetric pair of cycles of length $2k$}. 
A centrosymmetric cycle has even length and a centrosymmetric pair of cycles has length equal to a multiple of four. The edges of the digraph $\Gamma(Q)$ of a  centrosymmetric permutation matrix $Q$ are partitioned into centrosymmetric cycles and centrosymmetric pairs of cycles.

We first derive the following property of the extreme points of $\Omega_n^{t\& h}$ when $n$ is even. 

\begin{lemma}\label{lem:new}
Let $n$ be an even integer. Then every extreme point of $\Omega_n^{t\& h}$
is of the form
\[\frac{1}{2}(Q+Q^t)\]
where $Q$ is an $n\times n$ centrosymmetric permutation matrix.
\end{lemma}

\begin{proof} It follows from Lemma \ref{eq:genextremepoints} that every extreme point $A$ of $\Omega_n^{t\& h}$ is of the form
$A=\frac{1}{4}\left((P+P^{\pi})+(P+P^{\pi})^t\right)$ for some permutation matrix $P$. By Lemma \ref{lem:evencentroperm}, $P+P^{\pi}=Q_1+Q_2$ where $Q_1$ and $Q_2$ are centrosymmetric permutation matrices. Thus
\begin{equation}\label{eq:OK} A=\frac{1}{4}((Q_1+Q_2)+(Q_1+Q_2)^t)=\frac{1}{4}((Q_1+Q_1^t)+(Q_2+Q_2^t))\end{equation}
where $Q_1+Q_1^t$ and $Q_2+Q_2^t$ are symmetric and centrosymmetric, and thus Hankel symmetric as well. Hence if $Q_1+Q_1^t\ne Q_2+Q_2^t$, then $A$ is not an extreme point of $\Omega_n^{t\& h}$. 
If $Q_1+Q_1^t=Q_2+Q_2^t$, then by (\ref{eq:OK}), 
\[A=\frac{1}{4}(2(Q_1+Q_1^t))=\frac{1}{2}(Q_1+Q_1^t)\]
where $Q_1$ is a centrosymmetric permutation matrix.
\end{proof}

\begin{lemma}\label{lem:neven1} Let $n$ be an even integer.
Let $Q$ be an $n\times n$ centrosymmetric permutation matrix with a centrosymmetric cycle of length equal to a multiple of four. Then
$\frac{1}{2}(Q+Q^t)$  is not an extreme point of $\Omega_n^{t\& h}$.
\end{lemma}

\begin{proof} 
It suffices to assume that $Q$ is itself a centrosymmetric cycle of length $n$, a multiple of 4,   with corresponding cycle $i_1\rightarrow i_2\rightarrow\cdots\rightarrow i_n\rightarrow i_1$. Then $Q+Q^t=P_1+P_2$ where $P_1\ne  P_2$ are  symmetric, centrosymmetric (and thus Hankel symmetric) permutation matrices corresponding, respectively, to the permutations 
\[i_1\leftrightarrow i_2, i_3\leftrightarrow i_4,\ldots, i_{n-1}\leftrightarrow i_{n}
\mbox{ and } i_2\leftrightarrow i_3, i_4\leftrightarrow i_5,\ldots,i_{n}\leftrightarrow i_1.\]
Thus $\frac{1}{2}(Q+Q^t)=\frac{1}{2}(P_1+P_2)$. Since $P_1$ and $P_2$ are symmetric and Hankel symmetric, $\frac{1}{2}(Q+Q^t)$ is not an extreme point of $\Omega_n^{t\& h}$.
\end{proof}

\begin{lemma}\label{lem:neven2} Let $n$ be an even integer.
Let $Q$ be an $n\times n$ centrosymmetric permutation matrix with a centrosymmetric pair of cycles of length  equal to a multiple of four. Then
$\frac{1}{2}(Q+Q^t)$  is not an extreme point of $\Omega_n^{t\& h}$.
\end{lemma}

\begin{proof} The proof is similar to that of Lemma \ref{lem:neven1}. 
It suffices to assume that $Q$ is itself a pair of centrosymmetric cycles  of length $n=4k$,   corresponding to the disjoint pair of permutations \[i_1\rightarrow i_2\rightarrow\cdots\rightarrow i_{2k}\mbox{ and }
n+1-i_{1}\rightarrow n+1-i_{2}\rightarrow\cdots\rightarrow n+1-i_{2k}.\]
Then $Q+Q^t=P_1+P_2$ where $P_1$ and $P_2$ are  symmetric, centrosymmetric (and thus Hankel symmetric)  permutation matrices corresponding, respectively, to the  permutations
\[i_1\leftrightarrow i_2, \ldots,i_{2k-1}\leftrightarrow i_{2k}, n+1-i_1\leftrightarrow
n+1-i_2,\ldots,n+1-i_{2k-1}\leftrightarrow n+1-i_{2k}\]
and
\[i_2\leftrightarrow i_3,\ldots,i_{2k}\leftrightarrow i_1, n+1-i_2\leftrightarrow n+1-
i_{3}, n+1-i_{2k}\leftrightarrow n+1-i_1.\]
As above, $\frac{1}{2}(Q+Q^t)=\frac{1}{2}(P_1+P_2)$. Since $P_1$ and $P_2$ are symmetric and Hankel symmetric, $\frac{1}{2}(Q+Q^t)$ is not an extreme point of $\Omega_n^{t\& h}$.
\end{proof}

\begin{theorem}\label{th:evenextreme}
Let $n$ be an even integer. Then the extreme points of $\Omega_n^{t\& h}$ are the matrices of the form
$\frac{1}{2}(Q+Q^t)$ where $Q$ is a centrosymmetric 
permutation matrix not having any centrosymmetric cycles of length equal to a multiple of four nor centrosymmetric pairs of cycles  of length equal to a multiple of four.
\end{theorem}

\begin{proof}
Because of Lemmas \ref{lem:neven1} and \ref{lem:neven2}, we have only to show that if $Q$ is an $n\times n$  centrosymmetric permutation matrix  without any centrosymmetric cycles of length a multiple of four and centrosymmetric pairs of cycles of length a multiple of four, then $\frac{1}{2}(Q+Q^t)$  is an extreme point of $\Omega_n^{t\& h}$.
Let $Q$ be such a  centrosymmetric permutation matrix. 
That $\frac{1}{2}(Q+Q^t)$ is an extreme point of $\Omega_n^{t\& h}$ will follow by verifying that the only symmetric, Hankel-symmetric doubly stochastic matrix having zeros in all positions that $Q+Q^t$ has zeros
is $\frac{1}{2}(Q+Q^t)$. For this it suffices to show that this holds when our $Q$ is itself (i) a centrosymmetric  cycle $\gamma$ or (ii) a centrosymmetric pair of cycles $\{\gamma,\gamma^{\pi}\}$.

\smallskip
(i) In this case, for some integer $k$,  $Q$ is  a $(4k+2)\times (4k+2)$ centrosymmetric permutation matrix corresponding to a permutation of $\sigma$ of $\{1,2,\ldots,4k+2\}$ that is a cycle of  length $4k+2$.
 Then $Q^t$ is also such a centrosymmetric  permutation matrix (indeed, $(Q^t)^{\pi} =(Q^{\pi})^t=Q^t$), and  $(Q+Q^t)$ is a symmetric and Hankel symmetric $(0,1)$-matrix  with two 1's in each row and column,  whose associated digraph $\Gamma(Q+Q^t)$ consists of a cycle $\rho$ of length $4k+2$ and its reverse cycle in the other direction.  The symmetry, centrosymmetry, and the fact that the cycle $\gamma$ has length equal to 2 mod 4 now implies that $A$ has all its nonzero entries equal to $1/2$ and hence $A=\frac{1}{2}(Q+Q^t)$.

\smallskip
(ii) A similar argument works if we have a centrosymmetric pair of cycles of length equal to 2 mod 4, each cycle having length equal to 1 mod 2.
\end{proof}

Referring to Theorem \ref{th:evenextreme} we note that if $Q$ is a centrosymmetric permutation matrix, then $(Q+Q^t)^h=Q^h+Q^{\pi}=Q+Q^h$. Thus in Theorem \ref{th:evenextreme} we could replace $Q+Q^t$ by $Q+Q^{\pi}$.
As a final remark in the even case, we note that the characterization of the extreme points of $\Omega_n^{t\& h}$ includes the permutation matrices that are symmetric and Hankel-symmetric. This is because a symmetric permutation matrix can have cycles only of lengths 1 and 2.

We now turn to the investigation of the extreme points of $\Omega_n^{t\& h}$ when $n$ is odd. We begin with two examples which are useful for undertanding our general conclusions.

\begin{example}{\rm \label{ex:oddext1}
Let $n=11$ and consider the matrix  in ${\Omega}_9^{t\&h}$
given by
\[A=\frac{1}{4}\left[\begin{array}{c|c|c|c|c||c||c|c|c|c|c}
&3&&&1&&&&&&\\ \hline
3&&1&&&&&&&&\\ \hline
&1&&3&&&&&&&\\ \hline
&&3&&1&&&&&&\\ \hline
1&&&1&&2&&&&&\\ \hline\hline
&&&&2&&2&&&&\\ \hline\hline
&&&&&2&&1&&&1\\ \hline
&&&&&&1&&3&&\\ \hline
&&&&&&&3&&1&\\ \hline
&&&&&&&&1&&3\\ \hline
&&&&&&1&&&3&\end{array}\right]\]
whose  loopy graph $G(A)$ consists of two cycles of length 5 with a path of length 2 joining them and containing the central vertex 5 in its  middle.
It is straightforward to check that the only  matrix in ${\Omega}_{11}^{t\&h}$ which has a zero entry wherever $A$ has a zero entry (that is, whose graph is contained in the graph of $A$) is equal to  $A$. This implies that $A$ is an extreme point of ${\Omega}_{11}^{t\&h}$. This example extends to any odd $n\ge 5$ when the graph consists of two 
odd cycles of the same length and an even path joining them containing the central vertex $(n+1)/2$ corresponding to row and column $(n+1)/2$ in the middle.

Now let $n=9$ and consider the matrix in ${\Omega}_9^{t\&h}$
given by
\[A_1=\frac{1}{4}
\left[\begin{array}{c|c|c|c||c||c|c|c|c}
&&&&&&1&3&\\ \hline
&&&&&&1&&3\\ \hline
&&&&&2&&1&1\\  \hline
&&&&2&&2&&\\ \hline\hline
&&&2&&2&&&\\ \hline\hline
&&2&&2&&&&\\ \hline
1&1&&2&&&&\\ \hline
3&&1&&&&&&\\ \hline
&3&1&&&&&&\end{array}\right]\]
whose loopy graph $G(A_1)$ is a cycle of length $6$ bisected by a path of length 4 containing the central vertex 5 in its middle. 
Again, the only  matrix in ${\Omega}_9^{t\&h}$ which has a zero entry wherever $A_1$ has a zero entry  is equal to  $A_1$ and hence  $A_1$ is an extreme point of ${\Omega}_9^{t\&h}$. 
This example extends to any odd $n\ge 5$ when the graph consists of an even cycle bisected by a path of even length containing the central vertex $(n+1)/2$ in its middle. We note that if we r

Now let $n=7$ and consider the matrix in ${\Omega}_7^{t\&h}$
given by
\[A_2=\frac{1}{2}\left[\begin{array}{c|c|c||c||c|c|c}
&1&&&&&1\\ \hline
1&&1&&&&\\ \hline
&1&&1&&&\\ \hline\hline
&&1&&1&&\\ \hline\hline
&&&1&&1&\\ \hline
&&&&1&&1\\ \hline
1&&&&&1&\end{array}\right]\]
whose graph $G(A_2)$ is a cycle of length 7.
Again it is straightforward to check that any matrix in ${\Omega}_7^{t\&h}$ which has a zero entry wherever $A_2$ has a zero entry is equal to  $A_2$,  and hence $A_2$ is an extreme point of ${\Omega}_7^{t\&h}$. This example extends to any odd  $n\ge 3$ when the graph is  a cycle. Note that if the columns of $A_2$ are reversed in order, the corresponding loopy graph is of the type of $G(A_1)$.
As we shall see, these matrices $A_1$ and $A_2$ and their loopy graphs are the key for the classification of the extreme points of ${\Omega}_n^{t\&h}$ when $n$ is odd.
}\hfill{$\Box$}
\end{example}

\begin{lemma}\label{lem:even_to_odd}
Let $n=2m+1$ be an odd integer and let $A=[a_{ij}]$ be an extreme point of $\Omega_n^{t\& h}$ with $a_{m+1,m+1}=0$. Then 
one of the following holds:
\begin{itemize}
\item[\rm (a)] there exists an integer $k$ with $1\le k\le m$ such that \[a_{k,m+1}=a_{m+1,k}=a_{m+1,2m+2-k}=a_{2m+2-k,m+1}=\frac{1}{2},\]
\item[\rm (b)]  there exist integers $k$ and $l$  with $1\le k<l\le m$ such that 
\[a_{k,m+1}=a_{m+1,k}=a_{m+1,2m+2-k}=a_{2m+2-k,m+1}=\frac{1}{4}\] \mbox{ and }
\[
a_{l,m+1}=a_{m+1,l}=a_{m+1,2m+2-l}=a_{2m+2-l,m+1}=\frac{1}{4}.\]
\end{itemize}
In both instances, all other entries in row $m+1$ and column $m+1$ equal zero.
\end{lemma}

\begin{proof}
Since $a_{m+1,m+1}=0$, this is an immediate consequence of Corollary \ref{cor:quarter}  and the  symmetry and Hankel-symmetry of the doubly stochastic matrix $A$.
\end{proof}

The matrices in Example \ref{ex:oddext1} illustrate both possibilities in Lemma \ref{lem:even_to_odd}.

\begin{lemma}\label{lem:combine} Let $n=m+p$ be an integer where
$m$ is an odd integer and  $p=2q$ is an even integer. Let $B$ be an extreme point of $\Omega_m^{t\& h}$, and let 
\[C=\left[\begin{array}{c|c}
C_1&C_2\\ \hline
C_2^t&C_1^{h}\end{array}\right]\]
be an extreme point of $\Omega_p^{t\& h}$ where $C_i$ is $q\times q$ for $i=1$ and $2$. Then
\begin{equation}\label{eq:scattered}
A=\left[\begin{array}{ccc}
C_1&O_{qm}&C_2\\
O_{mq}&B&O_{mq}\\
C_2^t&O_{qm}&C_1^{h}\end{array}\right]\end{equation}
 is an extreme point of $\Omega_{n}^{t\& h}$. Conversely, if $A$ is an extreme point of $\Omega_{n}^{t\& h}$, then $C$ is an extreme point of $\Omega_p^{t\& h}$ and $B$ is an extreme point of $\Omega_m^{t\& h}$.
\end{lemma}

\begin{proof} The conclusion of the lemma is obvious.
\end{proof}

In Lemma \ref{lem:combine}, let $m=2k+1$. Then in (\ref{eq:scattered}) defining $A$, the matrix $B$ and $C$ may be `intertwined', more precisely, with indices $1\le i_1<i_2<\cdots<i_{k}<i_{k+1}=(n+1)/2$, the principal $m\times m$  submatrix  $A[\{i_1,i_2,\ldots,i_k,i_{k+1}, n+1-i_k,\ldots,n+1-i_2,n+1-i_1\}]$  of $A$ satisfies
\[A[\{i_1,i_2,\ldots,i_k,i_{k+1}, n+1-i_k,\ldots,n+1-i_2,n+1-i_1\}]=B,\]
where $C$ is placed in $A$ as a principal submatrix in the obvious way to produce a matrix in $\Omega_{n}^{t\& h}$. The loopy graph G$(A)$  is disconnected and consists of two vertex-disjoint loopy graphs G$(C)$ and G$(B)$. Since the matrices $C$ and $B$ are extreme points of  $\Omega_m^{t\& h}$ and $\Omega_p^{t\& h}$, respectively, $A$ will be an extreme point of $\Omega_n^{t\& h}$.
Conversely, if $A$ is an arbitrary extreme point of $\Omega_n^{t\& h}$ whose loopy graph G$(A)$ is not connected, then after a simultaneous permutation of rows and columns resulting in a matrix in $\Omega_n^{t\& h}$, $A$ will assume the form given in (\ref{eq:scattered}). Moreover, since the extreme points of $\Omega_p^{t\& h}$ for $p$ even are characterized in Theorem \ref{th:evenextreme}, it now follows from Lemma \ref{lem:combine} and the   comments above that to characterize the extreme points $A$ of $\Omega_{n}^{t\& h}$ when $n$ is odd,   it suffices to characterize  them when the corresponding loopy graph $G(A)$ is connected.

\begin{theorem}\label{th:oddchar}
Let $n$ be an odd integer. Then the connected component containing vertex $(n+1)/2$ of the graphs of the extreme points of 
$\Omega_n^{t\& h}$ are one of the following three types:
\begin{itemize}
\item[\rm (a)] An odd cycle $($possibly a loop$)$ containing vertex $(n+1)/2$.
\item[\rm (b)] Two odd cycles of the same length and an even path
$($possibly of length zero$)$ 
containing the central vertex $(n+1)/2$ in its center.
\item[\rm (c)] An even cycle bisected by an even path of nonzero length containing the central vertex $(n+1)/2$ in its middle.
\end{itemize}
Moreover, any matrix in $\Omega_n^{t\& h}$ whose graph is connected and satisfies one of {\rm (a)}, {\rm (b)}, and {\rm (c)} is an extreme point of
$\Omega_m^{t\& h}$.
\end{theorem}

\begin{proof}
Let $A$ be an extreme point of $\Omega_n^{t\& h}$. As already remarked, it  suffices to assume that the loopy graph G$(A)$ is connected. If G$(A)$ is a loop, then (a) holds. Now assume that G$(A)$ is not a loop. Then each vertex must have degree at least two since  
G$(A)$  is connected and the matrix is doubly stochastic.  Moreover, it follows from Lemma \ref{lem:genextremepoints} that the degree in G$(A)$ of the central vertex $(n+1)/2$ is either 2 or 4. The  Hankel symmetry of G$(A)$ implies that the edges containing the  vertex $(n+1)/2$ come in pairs:
$\{(n+1)/2,i\}$ and $\{(n+1)/2,n+1-i\}$ where $i<(n+1)/2$. Consider such a pair of edges and follow the two paths $\pi_i$ and $\pi_{n+1-i}$ that they originate from vertex $(n+1)/2$ whereby if 
$\{k,l\}$ is an edge of $\pi_1$, then $\{n+1-k,n+1-l\}$ is an edge of $\pi_{n+1-i}$. The following possibilities can occur:
\begin{itemize}
\item[\rm (a)] $\pi_i$ and $\pi_{n+1-i}$ first meet at a common edge,  thereby creating a  cycle $\gamma$  of odd length containing vertex $(n+1)/2$.
\item[\rm (b)] (a) does not occur, and $\pi_i$  returns to a previous vertex $k$ on its path thereby creating a cycle $\gamma_i$, and  so $\pi_{n+1-i}$
returns to the previous vertex $(n+1-k)$  on its path thereby creating a cycle $\gamma_{n+1-i}$. If these two cycles have even length, then by alternatingly adding and subtracting a small number $\epsilon$ to  the entries corresponding to the edges of the cycles, we obtain two
distinct matrices $B_1$ and $B_2$ in $\Omega_n^{t\&h}$ such that
$A=\frac{1}{2}(B_1+B_2)$ implying that $A$ is not an extreme point. Hence these two cycles have odd length and are joined by a path of even length  containing the central vertex $(n+1)/2$ in its center. If $k=(n+1)/2$, then this path has length equal to zero.
\item[\rm (c)] Before either (a) or (b) occurs,  $\pi_i$ contains the vertex $(n+1-i)$ and so $\pi_{n+1-i}$ contains the vertex $i$, thereby creating an even length  cycle,  with an even length path joining vertices $i$ and $(n+1-i)$ having the central vertex $(n+1)/2$ in its center. 
\end{itemize}

Note that since G$(A)$ is connected and each vertex has degree at least 2, one of the possibilities (a), (b), and (c) must occur.
If $G^*$ is the loopy graph determined by (a), (b), or (c), then as illustrated in Example \ref{ex:oddext1}, there is a matrix $C$ in $\Omega_n^{t\& h}$ whose loopy graph is $G$ and $C$ is an extreme point of $\Omega_n^{t\& h}$. Suppose that the loopy graph G$(A)$ does not equal $G^*$. Then $A$ has a nonzero entry in a position which is zero in $C$. For $\epsilon>0$ sufficiently small, we have
\[\frac{1}{1+\epsilon} (A+\epsilon C),\
\frac{1}{1-\epsilon} (A-\epsilon C)\in \Omega_n^{t\& h},\]
where $A\ne A\pm \epsilon C$ implying the contradiction that $A$ is not an extreme point. Hence the loopy graph of $A$ equals $G$, and one of (a), (b), and (c) holds.
\end{proof}

Theorem \ref{th:oddchar} and Lemma \ref{lem:combine}, along with the characterization in Theorem \ref{th:evenextreme} of the extreme points of  $\Omega_m^{t\& h}$  when  $m$ is even now furnish a complete characterization of the extreme points of $\Omega_n^{t\& h}$ 
when $n$ is even.


\begin{thebibliography}{99}
\bibitem{BBS} M.~Barnabei, F.~Bonetti, and M.~Silimbani, The Eulerian distribution on centrosymmetric involutions, {\it Discrete Math. Theor. Comput, Sci.}, 11 (2009), 95-115.
\bibitem{RAB} R.A.~Brualdi, {\it Combinatorial Matrix Classes}, Cambridge University Press, 2006.
\bibitem{BF} R.A.~Brualdi and E.~Fritscher, Loopy, Hankel, and combinatorially skew-Hankel tournaments, {\it Discrete Applied. Math.}, 194(2015), 37--59.
\bibitem{BM} R.A.~Brualdi and Shi-Mei Ma, Centrosymmetric, and symmetric and Hankel-symmetric matrices, {\it Mathematics Across Contemporary Science}, T.~Abualrub et al. (eds), Springer Proceedings in Mathematics and Statistic, 190, 2017, 17--31.
\bibitem{BMe} R.A.~Brualdi and S.A.~Meyer, Combinatorial properties of integer matrices and integer matrices mod $k$, {\it Lin. Multilin. Alg.} to appear.
\bibitem{CN} Soojin Nam and Yunsun Nam, Convex polytopes of generalized doubly stochastic matrices, {\it Comm. Korean Math, Soc.}, 16 (2001), No. 4, 679--690.
\bibitem{AC} A.B.~Cruse, Some combinatorial properties of centrosymmetric matrices, {\it Linear Alg. Applics.}, 16(1977), 65--77.
\end{thebibliography}
\end{document}